\documentclass[11pt]{article}
\usepackage{a4,amsmath,amsfonts,amssymb,epsfig}
\usepackage{color,epsf}
\usepackage{latexsym}
\usepackage{graphicx}
\usepackage{amsthm}

\usepackage{natbib}

\usepackage{stix}

\usepackage{subfig}
\usepackage{alltt}
\usepackage{orcidlink}
\usepackage{mathtools}

\newcommand{\e}{\ensuremath{\mathrm{e}}}
\newcommand{\ad}{\ensuremath{\mathrm{ad}}}

\newtheorem{lemma}{Lemma}[section]

\newtheorem{theorem}[lemma]{Theorem}
\newtheorem{proposition}[lemma]{Proposition}
\newtheorem{corollary}[lemma]{Corollary}
\newtheorem{remark}[lemma]{Remark}

\begin{document}

\title{Error bounds for splitting methods in unitary problems}

\author{F. Casas$^{1 \orcidlink{0000-0002-6445-279X}}$, A. Murua$^{2 \orcidlink{0000-0002-3430-9642}}$ 
  \\[2ex]
$^{1}$ {\small\it Departament de Matem\`atiques and IMAC, Universitat Jaume I}, \\{\small\it 12071-Castell\'on, Spain}\\{
\small\it Email: Fernando.Casas@mat.uji.es}\\[1ex]
$^{2}$ {\small\it Konputazio Zientziak eta A.A. Saila, Informatika Fakultatea, EHU/UPV}, \\
 {\small\it Donostia/San Sebasti\'an, Spain}\\{
\small\it Email: Ander.Murua@ehu.es}\\[1ex]
}

\maketitle

\begin{abstract}

Splitting methods constitute a widely used class of numerical integrators for ordinary and partial differential equations, particularly well suited to problems that can be decomposed into simpler subproblems. High-order splitting schemes are available that achieve high accuracy while preserving key qualitative properties of the underlying dynamical system, and are successfully used across a broad range of fields. In this work, we present a systematic analysis of both local and global errors arising from arbitrary splitting methods applied to unitary problems. Two complementary types of error estimates are derived. The first is expressed in terms of operator norms, while the second is formulated using norms of commutators and can, under suitable assumptions, be extended to certain classes of unbounded operators. Special attention is devoted to the case where only two operators are involved. The theoretical results are illustrated by deriving explicit error bounds for some representative schemes.

\end{abstract}

\bigskip


\section{Introduction}

Splitting methods constitute a natural option for getting approximations to the solution of differential equations when they
can be subdivided into different sub-problems easier to solve than the original system \cite{mclachlan02sm}. 
Their range of applications covers a wide spectrum:
numerical analysis of differential equations, sampling of probability distributions in statistics and molecular dynamics
with Hamiltonian Monte Carlo methods, Vlasov--Poisson equations in plasma physics and quantum Monte Carlo methods, 
to name just a few \cite{blanes24smf}. They are also popular in the simulation of quantum systems in digital quantum computers, where they are called
\emph{product formulas} \cite{childs19nol,chen22epf}, and constitute a paradigmatic example of geometric numerical integrators \cite{hairer06gni}, since they preserve by construction structural
properties of the exact solution, such as symplecticity (in classical Hamiltonian dynamics) and unitarity (in quantum evolution problems) \cite{blanes25aci}.

To establish the framework of our discussion, we start by considering the linear equation
\begin{equation}  \label{basic_eq}
  \frac{du}{dt} =  A u = \sum_{j = 1}^N A_j  \, u, \qquad u(0) = u_0, 
\end{equation}
where $A_1, \ldots, A_N$ are bounded linear operators acting on a Hilbert space $\mathcal{H}$, $u(t) \in \mathcal{H}$ for all $t$ and
$\|A_j\| = a_j < \infty$ for $j=1,2,\ldots,N$. 

The exact solution of \eqref{basic_eq}, $u(t) = \e^{t (A_1 + \cdots + A_N)} u_0$, is supposed to be well defined for all $t \ge 0$. Associated
with equation \eqref{basic_eq}, we have
\begin{equation} \label{eq_op1}
  \frac{d }{dt} U(t) = A \, U(t), \qquad U(0) = I,
\end{equation}
in the sense that $u(t) = U(t) u_0$ and
\[
    U(t) = \e^{t (A_1 + \cdots + A_N)}
\]
satisfies \eqref{eq_op1}. It often happens that evaluating the action of the evolution operator $U(t)$ on $u_0$ is a difficult or computationally expensive task,
whereas computing each $\e^{t A_j}$ separately is much more feasible. In that case $u(t)$ can be approximated  by
$\e^{t A_1}  \e^{t A_2} \cdots \e^{t A_N} u_0$. In practice, however, this approximation performs poorly unless $t > 0$
is sufficiently small. Therefore, when the goal is to approximate the solution of \eqref{basic_eq} at a not-so-small final time $t = t_f$
the standard approach is to divide the interval $[0, t_f]$ into $K$ subintervals of length $h$,
with $K h = t_f$, and then compute the sequence
\begin{equation} \label{lie-trotter}
  u_{k+1} =  \e^{h A_1}  \e^{h A_2} \cdots \e^{h A_N} \, u_k, \qquad k \ge 0, 
\end{equation}  
(or any other permutation of the operators $A_j$)
so that $u_{k+1} \approx u(t_{k+1})$ with $t_{k+1} = (k+1)h$.  This corresponds to the Lie--Trotter scheme, and a direct calculation shows that
\[
  \e^{h A_1}  \e^{h A_2} \cdots \e^{h A_N} - \e^{h (A_1 + \cdots + A_N)} = \mathcal{O}(h^2) \qquad \mbox{ as } \qquad h \rightarrow 0,
\]
so that it is of first order of accuracy.  A higher-order approximation can be achieved with a symmetrized version of \eqref{lie-trotter},
\begin{equation} \label{strang}
  S_2(h) = \e^{\frac{1}{2} h A_1} \,  \e^{\frac{1}{2} h A_2} \, \cdots \,  \e^{h A_N} \, \cdots \e^{\frac{1}{2} h A_2} \, \e^{\frac{1}{2} h A_1}.
\end{equation}
This corresponds to the Strang splitting, and verifies
\[
     S_2(h) - \e^{h (A_1 + \cdots + A_N)} = \mathcal{O}(h^3) \qquad \mbox{ as } \qquad h \rightarrow 0.
\]
The more general composition
\begin{equation} \label{gen.split}
  \Psi(h) = \prod_{j=1}^s \e^{h \, \alpha_N^{(j)} A_N} \, \cdots \, \e^{h \, \alpha_2^{(j)} A_2}  \, \e^{h \alpha_1^{(j)} A_1},
\end{equation}
can achieve a given order $p$ 
if the number $s$ and the coefficients $ \alpha_i^{(j)}$ are appropriately chosen so that 
\begin{equation} \label{diff1}
     \Psi(h) - \e^{h (A_1 + \cdots +A_N)} = \mathcal{O}(h^{p+1}) \qquad \mbox{ as } \qquad h \rightarrow 0.
\end{equation}
 In particular, a necessary and sufficient condition for order $p= 1$ is
\begin{equation}
\label{eq:consistency}
	\sum_{j=1}^s \alpha_i^{(j)}=1, \qquad 1 \leq i \leq N.
\end{equation}
For $p>1$, the parameters $ \alpha_i^{(j)}$ have to satisfy an additional set of algebraic equations (the so-called order conditions) whose
number and complexity growths exponentially with $p$ \cite{blanes24smf}. For this reason, very often 
compositions of the basic Lie--Trotter scheme \eqref{lie-trotter} or the Strang splitting \eqref{strang} are considered instead of the direct approach \eqref{gen.split}, since
then the number of order conditions is significantly reduced.  More specifically:

\begin{itemize}
\item \textbf{Compositions of the basic Lie--Trotter scheme:}  
Let $\chi(h)$ denote the basic Lie--Trotter scheme:
\begin{equation}
\label{eq:chi}
\chi(h) = \mathrm{e}^{h A_1} \cdots \mathrm{e}^{h A_N},
\end{equation}
and $\chi^*(h)$ its adjoint:
\begin{equation}
\label{eq:chi*}
 \chi^*(h) = \chi(-h)^{-1} = \mathrm{e}^{h A_N} \cdots \mathrm{e}^{h A_1},
\end{equation}
both of which provide first-order approximations. One can then construct higher-order methods by taking compositions of the form
\begin{equation} \label{eq:Psi_alpha}
 \Psi(h) = \chi(\alpha_{2s} h) \, \chi^*(\alpha_{2s-1} h) \cdots \chi(\alpha_2 h) \, \chi^*(\alpha_1 h),
\end{equation}
where the coefficients $\alpha_1, \ldots, \alpha_{2s} \in \mathbb{R}$ are appropriately chosen so as to satisfy the required order conditions.

\item \textbf{Compositions of the Strang splitting:}  
Define
\begin{eqnarray} \label{eq:comp_Strang}
  \Psi(h) & =  &\chi\left(\frac{\gamma_s}{2} h\right) \chi^*\left(\frac{\gamma_s}{2} h\right) \cdots \chi\left(\frac{\gamma_1}{2} h\right) \chi^*\left(\frac{\gamma_1}{2} h\right)\nonumber \\
     & = & S_2(\gamma_s h) \cdots S_2(\gamma_1 h),
\end{eqnarray}
with $\gamma_1, \ldots, \gamma_s \in \mathbb{R}$.  
This is a particular case of composition \eqref{eq:Psi_alpha}, with coefficients $\alpha_1, \ldots, \alpha_{2s}$  given by
\[
   (\alpha_1, \alpha_2, \ldots, \alpha_{2s-1}, \alpha_{2s}) = \left( \frac{\gamma_1}{2}, \frac{\gamma_1}{2}, \ldots, \frac{\gamma_s}{2}, \frac{\gamma_s}{2} \right).
\]
\end{itemize}

The well-known triple-jump \cite{yoshida90coh} and quintuple-jump \cite{suzuki91gto} methods of arbitrary order $p$ are specific instances of  \eqref{eq:comp_Strang}. However, for orders $p \ge 6$, alternative choices of coefficients lead to more efficient schemes~\cite{blanes24smf}. These are typically obtained
by minimizing in a certain sense the main term in the local truncation error, that is, the term multiplying $h^{p+1}$ in \eqref{diff1}. In the analysis, 
terms multiplying $h^q$ with $q > p+1$ are generally ignored based on the implicit assumption that, for sufficiently small $h$,  their contribution to the error can be neglected
in comparison with the leading error term. Nevertheless, there are relevant physical examples where, for a fixed $h$, this is no longer the case. These
include near-neighbor lattice Hamiltonians \cite{childs19nol} and some quantum chemistry systems \cite{wecker14gce}. It is then much more helpful
to have rigorous (and ideally sharp) bounds for the error committed both when approximating $\e^{h A}$ with $\Psi(h)$, that is, for
\[
  \mathcal{E}_{\ell} \equiv \|\Psi(h) - \e^{h\, (A_1+\cdots+A_N)} \|
\]
and for the global error after $k$ iterations of the map, 
\[
  \mathcal{E}_g = \|\Psi(h)^k - \e^{k h \, (A_1+\cdots+A_N)} \|
\]
for an arbitrary positive integer $k$. Having explicit bounds for $\mathcal{E}_{\ell}$ and  $\mathcal{E}_g$ is indeed of interest in several contexts. 
In particular, in the
numerical analysis of differential equations, they can be used either to analyze the relative efficiency of different integrators by comparing
the respective bounds obtained or to construct new schemes by minimizing the bounds. Also,  in the 
Hamiltonian simulation in quantum computers by product formulas, minimizing the bound may lead to a reduction in complexity of the required
quantum circuit implementing the simulation \cite{childs21tot}.

Previous analyses of error bounds for splitting methods go back to the work of Suzuki \cite{suzuki85dfo}, where optimal bounds were obtained for the 
Lie--Trotter \eqref{lie-trotter} and Strang splitting \eqref{strang} in terms of the norm of commutators of the $A_j$ operators. The analysis for the
Strang splitting was later extended  to unbounded operators in \cite{jahnke00ebf} and  to splitting methods
of arbitrary order $p$ in \cite{thalhammer08hoe,descombes10ael}. An elementary approach for low-order schemes can be found in \cite{iserles24aea}, whereas in \cite{childs21tot} a general theory is presented leading to estimates of the form
\[
   \mathcal{E}_{\ell} = \mathcal{O} \left(  \tilde{\alpha}_{\mathrm{comm}} \, t^{p+1} \, \e^{4 t s \sum_{j=1}^N a_j}\right), 
\]
for a splitting method \eqref{gen.split} of order $p$ and $s$ stages. Here
$\tilde{\alpha}_{\mathrm{comm}}$ are bounds on the norm of nested commutators involving $p+1$ operators $A_j$. 

In this work we carry out a systematic treatment of the error committed by splitting methods of arbitrary order $p$, providing two types of estimates.
The first type is formulated in terms of the norm $a_j$ of the operators $A_j$ involved. 
In the second, we get explicit bounds depending on the norm of commutators
of the operators, generalizing the approach of  \cite{iserles24aea} and sharpening the analysis of \cite{childs21tot}. These results are
particularly relevant when the operators nearly commute and can in principle be generalized, under certain assumptions, to unbounded operators.

For simplicity, in the rest of the paper we will assume that the operators $A_j$ in \eqref{basic_eq} are skew-adjoint, that is, $A_j^{\dagger} \equiv A_j^*
= -A_j$, so that 
$U(t) = \e^{t A}$ is a unitary operator for $t \in \mathbb{R}$, $U^{*}(t) U(t) = I$, and $\|U(t)\| = 1$. This encompasses the important case
of quantum Hamiltonian problems.

\section{Error bounds depending on the norm of the operators}
\label{sec.2}

Our approach for obtaining error bounds in terms of the norm of the operators $A_j$ consists in identifying a suitable representation for both the exact solution of eq. \eqref{basic_eq}
and the numerical approximation  rendered by the splitting method,
and then estimating the difference between them. We first establish a general result and then apply it to several
types of compositions.


\subsection{General estimates of the local error}

The most general splitting scheme \eqref{gen.split} can be written in the form
\begin{equation}
\label{eq:expB}
\Psi(h) = \mathrm{e}^{h C_r} \cdots \mathrm{e}^{h C_1} = I + \sum_{n \geq 1} h^n \sum_{j_1+\cdots+j_r = n} \frac{C_r^{j_r} \cdots C_1^{j_1}}{j_r! \cdots j_1!},
\end{equation}
with $r= s N$, and
\[
C_{i + N (j-1)} = \alpha_i^{(j)} \, A_i, \qquad 1\leq i \leq N, \qquad 1\leq j \leq s.
\]

If the consistency condition \eqref{eq:consistency} holds, then 
$C_1 + \cdots + C_r = A_1 + \cdots + A_N$, and the local error $\Psi(h) - \mathrm{e}^{h(A_1 + \cdots + A_N)}$ of the scheme can be written simply as
\begin{equation*}
\mathrm{e}^{h C_r} \cdots \mathrm{e}^{h C_1} - \mathrm{e}^{h (C_r+\cdots+C_1)}.
\end{equation*}

\begin{proposition}
\label{p:boundsB}
If $C_1, \ldots, C_r$ are bounded skew-adjoint operators, then, for all $q \geq 1$,
\begin{equation}
\label{eq:expB2}
\mathrm{e}^{h C_r} \cdots \mathrm{e}^{h C_1} = I + \sum_{n=1}^{q} h^n \sum_{j_1+\cdots+j_r = n} \frac{C_r^{j_r} \cdots C_1^{j_1}}{j_r! \cdots j_1!} + h^{q+1} \mathcal{R}_{q+1}(h),
\end{equation}
where
\begin{equation*}
\|\mathcal{R}_{q+1}(h)\| \leq \frac{\big(\|C_1\| + \cdots + \|C_r\| \big)^{q+1}}{(q+1)!}.
\end{equation*}
\end{proposition}

\begin{proof}
Consider the initial value problem
\begin{equation}
\label{ivp.1}
\frac{d}{d\tau} Y(\tau) = h A(\tau) Y(\tau), \qquad Y(0) = I,
\end{equation}
where
\[
A(\tau) = C_j \quad \text{for } \quad \tau \in [j-1, j], \qquad j=1,\ldots,r.
\]
The solution of \eqref{ivp.1} at $\tau = r$ is
\begin{equation} \label{solu_Y}
Y(r) = \mathrm{e}^{h C_r} \cdots \mathrm{e}^{h C_1}.
\end{equation}
Since $A^\dagger(\tau) = -A(\tau)$, the solution operator $Y(\tau)$ is unitary for all $\tau$ and can be expanded as a power series in $h$ (cf. \cite{blanes24smf}),
\[
Y(\tau) = I + \sum_{n=1}^{q} h^n \int_0^{\tau} \int_0^{\tau_1} \cdots \int_0^{\tau_{n-1}} A(\tau_1) \cdots A(\tau_n) \, d\tau_n \cdots d\tau_1 
 \, + h^{q+1} \hat{\mathcal{R}}_{q+1}(\tau, h),
\]
where for all $n$,
\[
\hat{\mathcal{R}}_n(\tau, h) = \int_0^{\tau} \cdots \int_0^{\tau_{n-1}} A(\tau_1) \cdots A(\tau_n) Y(\tau_n) \, d\tau_n \cdots d\tau_1.
\]
Thus, by expanding the exponentials in \eqref{solu_Y}, we get \eqref{eq:expB2}, with
\[
\mathcal{R}_{q+1}(h) = \hat{\mathcal{R}}_{q+1}(r, h).
\]
Let us determine next a bound for $\mathcal{R}_n(h)$.  Since $Y(\tau)$ is unitary, then
\[
\|\hat{\mathcal{R}}_n(\tau,h)\| \leq \int_0^{\tau} \cdots \int_0^{\tau_{n-1}} \|A(\tau_1)\| \cdots \|A(\tau_n)\| \, d\tau_n \cdots d\tau_1.
\]
Define now $a(\tau) = \|A(\tau)\|$, so that
\[
 \|\mathcal{R}_n(h)\|  = \|\hat{\mathcal{R}}_n(r, h)\| \leq \int_0^r \cdots \int_0^{\tau_{n-1}} a(\tau_1) \cdots a(\tau_n) \, d\tau_n \cdots d\tau_1.
\]
Next, consider the scalar initial value problem
\[
\frac{d}{d\tau} y(\tau) = h a(\tau) y(\tau), \qquad y(0) = 1,
\]
whose solution admits the expansion
\begin{equation} \label{expre_1yt}
y(\tau) = 1 + \sum_{n \geq 1} h^n \int_0^\tau \cdots \int_0^{\tau_{n-1}} a(\tau_1) \cdots a(\tau_n) \, d\tau_n \cdots d\tau_1.
\end{equation}
Since
\[
a(\tau) =\|C_j\| \quad \text{ for } \;\; \tau \in [j-1, j], \qquad j = 1, \ldots, r,
\]
we have
\begin{equation} \label{expre_2_yr}
y(r) = \e^{h( \|C_1\|  + \cdots + \|C_r\| )} = 1 + \sum_{n \geq 1} h^n \frac{(\|C_1\| + \cdots + \|C_r\| )^n}{n!}.
\end{equation}
Comparing \eqref{expre_2_yr} with expression \eqref{expre_1yt} evaluated at $\tau = r$, we obtain the bound
\[
\|\mathcal{R}_n(h)\| \leq \frac{(\|C_1\| + \cdots + \|C_r\| )^n}{n!},
\]
as claimed.
\end{proof}

\begin{corollary}
\label{c:expB_le}
Under the assumptions of Proposition~\ref{p:boundsB}, for all $q\geq 1$,
\begin{align*}
 & \|\mathrm{e}^{h C_r} \cdots \mathrm{e}^{h C_1} - \mathrm{e}^{h (C_1+\cdots + C_r)}\| \leq 
 \sum_{n=1}^{q} h^n \left\| \sum_{j_1+\cdots+j_r = n} \frac{C_r^{j_r} \cdots C_1^{j_1}}{j_r! \cdots j_1!} 
 - \frac{(C_1+\cdots+C_r)^n}{n!} \right\| \\
 & \qquad\qquad  + \frac{h^{q+1}}{(q+1)!} \left( ( \|C_1\|+\cdots + \|C_r\|)^{q+1} + \|C_1+\cdots + C_r\|^{q+1} \right).
\end{align*}
\end{corollary}
Application of Corollary~\ref{c:expB_le} to the general splitting scheme \eqref{gen.split} with $r=s\, N$ and
\begin{equation} \label{def_Bs}
C_{i + N (j-1)} = \alpha_i^{(j)} \, A_i, \qquad 1\leq i \leq N, \qquad 1\leq j \leq s,
\end{equation}
leads directly to the following result.

\begin{theorem} 
\label{th_gen.split}
Suppose that the product formula 
\[
    \Psi(h) = \prod_{j=1}^s \e^{h \, \alpha_N^{(j)} A_N} \, \cdots \, \e^{h \, \alpha_2^{(j)} A_2}  \, \e^{h \alpha_1^{(j)} A_1},
\]
with $\alpha_i^{(j)} \in \mathbb{R}$, provides an approximation of order $p$ to $\e^{h(A_1 + \cdots + A_N)}$,
with $A_j$ skew-adjoint operators verifying that $\|A_j\| =  a_j < \infty$ for $j=1,\ldots, N$. Then, the following estimate holds:
\begin{equation*}
  \left\|\Psi(h)-\e^{h \sum_{j=1}^N A_j} \right\| \leq \frac{h^{p+1}}{(p+1)!} \left( \left[\sum_{i=1}^N a_i \right]^{p+1} + \left[ \sum_{i=1}^N a_i \, \left(|\alpha_i^{(1)}|+\cdots+|\alpha_i^{(s)}|\right)\right]^{p+1} \right)
  \end{equation*}
\end{theorem}
\begin{proof}
Taking into account \eqref{def_Bs}, the consistency condition \eqref{eq:consistency}, and the fact that $\Psi(h)$ is of order $p$, we have 
\[
\begin{aligned}
   \left\|\Psi(h)-\e^{h \sum_{j=1}^N A_j} \right\| & = \|\mathrm{e}^{h C_r} \cdots \mathrm{e}^{h C_1} - \mathrm{e}^{h (C_1+\cdots + C_r)}\| \\
   & \le
   \frac{h^{p+1}}{(p+1)!} \left( ( \|C_1\|+\cdots + \|C_r\|)^{p+1} + \|C_1+\cdots + C_r\|^{p+1} \right),
\end{aligned}   
\]   
but $\|C_1+\cdots + C_r\| = \|A_1+\cdots + A_N\| \le a_1 + \cdots + a_N$ and
\[
  \|C_1\|+\cdots + \|C_r\| = \sum_{i=1}^N  |\alpha_i^{(1)}| \, a_i + \cdots + \sum_{i=1}^N  |\alpha_i^{(s)}| \, a_i.
\]  
\end{proof}

\subsection{Estimates for compositions of the Lie--Trotter and Strang schemes}

The previous analysis can be particularized when the product formula is a composition of Lie--Trotter and Strang schemes.
Specifically, composition \eqref{eq:Psi_alpha} can be expressed as a product \eqref{eq:expB} of $r = 2Ns$
bounded operators $C_1, \ldots, C_r$,
\[
  \Psi(h) = \e^{h C_r} \, \cdots \, \e^{h C_1},
\]
with
\[
  C_{(2j-2)N + k} = \alpha_{2j-1} A_k, \qquad  C_{(2j-1)N + k} = \alpha_{2j} A_{N-k+1}, \qquad k=1,\ldots,N, \quad j=1, \ldots, s.
\]
In consequence, the norms $ \| C_j \|$ are given by
\[
\begin{array}{l}
  \| C_{(2j-2)N + k}\| = |\alpha_{2j-1}| \, a_k, \\
    \|C_{(2j-1)N + k}\| = |\alpha_{2j}| \, a_{N-k+1}, 
 \end{array}   
    \qquad j=1, \ldots, s; \;  k=1,\ldots,N,
\]
so that
\[
\sum_{j=1}^r \|C_j\| = \sum_{j=1}^s \sum_{k=1}^N \Big( |\alpha_{2j-1}| \, a_k + |\alpha_{2j}| \, a_{N-k+1} \Big) = 
  \left( \sum_{j=1}^{2s} | \alpha_j| \right) \sum_{k=1}^N a_k.
\]
Therefore, in virtue of Corollary \ref{c:expB_le}, we have the following result.
\begin{theorem} \label{th_compoLT}
Suppose that the $2s$-stage composition 
\[
\Psi(h) = \chi(\alpha_{2s} h) \, \chi^*(\alpha_{2s-1} h) \cdots \chi(\alpha_2 h) \, \chi^*(\alpha_1 h),
\]
with $(\alpha_1, \ldots, \alpha_{2s}) \in \mathbb{R}^{2s}$, and
\[
  \chi(h) = \e^{h A_1} \, \e^{h A_2} \cdots \e^{h A_N}, \qquad \chi^*(h) = \e^{h A_N} \cdots \e^{h A_2} \, \e^{h A_1}, 
\]
 provides an approximation of order $p$ to $\e^{h(A_1 + \cdots + A_N)}$,
with $A_j$ skew-adjoint operators verifying that $\|A_j\| = a_j < \infty$ for $j=1,\ldots, N$. Then, the following estimate holds:
\begin{equation}  \label{eq_a1}
  \left\|\Psi(h)-\e^{h \sum_{j=1}^N A_j} \right\| \le \frac{1}{(p+1)!} \left[ 1 +  \left( \sum_{i=1}^{2s} |\alpha_i| \right)^{p+1} \right] \, \left(h \sum_{j=1}^N a_j \right)^{p+1}.
\end{equation}
\end{theorem}
Taking into account that the composition \eqref{eq:comp_Strang} of the Strang splitting constitutes a particular example of \eqref{eq:Psi_alpha}, we get an analogous
error estimate in this case.
\begin{corollary} \label{th_compoS}
Suppose that the $s$-stage composition 
\[
  \Psi(h) = S_2(\gamma_s h) \,  \cdots \, S_2(\gamma_2 h) \, S_2(\gamma_1 h)
\]
with $(\gamma_1, \ldots, \gamma_s) \in \mathbb{R}^s$, and $S_2(h) = \chi(h/2) \chi^*(h/2)$ with
\[
  \chi(h) = \e^{h A_1} \, \e^{h A_2} \cdots \e^{h A_N}, \qquad \chi^*(h) = \e^{h A_N} \cdots \e^{h A_2} \, \e^{h A_1}, 
\]
 provides an approximation of order $p$ to $\e^{h(A_1 + \cdots + A_N)}$,
and $A_j$ skew-adjoint operators verifying that $\|A_j\| = a_j < \infty$ for $j=1,\ldots, N$. Then, the following estimate holds:
\begin{equation}  \label{eq_a2}
  \left\|\Psi(h)-\e^{h \sum_{j=1}^N A_j} \right\| \le \frac{1}{(p+1)!} \left[ 1 +  \left( \sum_{i=1}^s |\gamma_i| \right)^{p+1} \right] \, \left(h \sum_{j=1}^N a_j \right)^{p+1}.
\end{equation}
\end{corollary}
It is worth remarking that bounds \eqref{eq_a1} and \eqref{eq_a2} depend solely on the 1-norm of the coefficient vectors $(\alpha_1, \ldots, \alpha_{2s})$ and
$(\gamma_1, \ldots, \gamma_s)$, respectively. This observation provides some theoretical ground for the common practice  when designing of numerical schemes, namely minimizing this 1-norm for getting more efficient integrators \cite{blanes24smf}.

\subsection{Estimates of the local error for symmetric schemes}
\label{subsec_2.3}

Composition schemes \eqref{eq:Psi_alpha} and \eqref{eq:comp_Strang} with a palindromic sequence of coefficients, that is,
\begin{equation} \label{palindromic_alpha}
  \Psi(h) = \chi(\alpha_{2s} h) \, \chi^*(\alpha_{2s-1} h) \cdots \chi(\alpha_2 h) \, \chi^*(\alpha_1 h), 
     \quad \mbox{ with } \quad \alpha_{2s+1-j} = \alpha_j, \quad 1 \leq j \leq 2s
\end{equation}
and
\begin{equation}
\label{eq:palindromic_gamma}
\Psi(h) = S_2(\gamma_s h) \,  \cdots \, S_2(\gamma_2 h) \, S_2(\gamma_1 h) \qquad \mbox{ with } \qquad \gamma_{s+1-i} = \gamma_i, \quad 1 \leq i \leq s,
\end{equation}
respectively, are a common choice for high-order methods, as they preserve time symmetry and typically offer a favorable trade-off between the number of maps in the composition and the resulting accuracy.
It turns out that the previous error estimates can also be improved for these cases, as we next show.

The mapping \eqref{palindromic_alpha} admits a symmetric factorization,
\begin{equation} \label{eq:sym_factorization}
\Psi(h) = \Phi^*(h) \, \Phi(h),
\end{equation}
where
$\Phi^*(h) = \Phi(-h)^{-1}$, and
\begin{equation}
\label{eq:Phi_s_odd}
\Phi(h) = \chi^*(\alpha_{s} h) \, \chi(\alpha_{s-1} h) \, \cdots \, \chi(\alpha_2 h) \, \chi^*(\alpha_1 h) \qquad \mbox{ for $s$ odd}, 
\end{equation}
and
\begin{equation}
\label{eq:Phi_s_even}
\Phi(h) = \chi(\alpha_{s} h) \, \chi^*(\alpha_{s-1} h)  \, \cdots \, \chi(\alpha_2 h) \, \chi^*(\alpha_1 h) \qquad \mbox{ for $s$ even}. 
\end{equation}
We will make use of the following auxiliary result for arbitrary approximations $\Psi(h)$ of $\e^{h (A_1+\cdots+A_n)}$ that admit the symmetric factorization \eqref{eq:sym_factorization}:

\begin{lemma} \label{l:Theta}
Let  $\Phi(h)$ be a one-parameter family of unitary operators, and consider the symmetric factorization $\Psi(h) = \Phi^*(h) \, \Phi(h)$.
Then
\begin{equation*}
\|\Psi(h) - \e^{h\,  (A_1+\cdots+A_N)}\| \leq \| \Theta(h) - \Theta(-h)\|,
\end{equation*}
where
\begin{equation} \label{eq:Theta}
\Theta(h) = \Phi(h) \, \e^{-\frac{h}{2}  (A_1+\cdots+A_N)}.
\end{equation}
\end{lemma}
\begin{proof}
The result follows from the identity
\begin{equation*}
\Psi(h) - \e^{h\,  (A_1+\cdots+A_N)} = \Phi^*(h) \, \big(\Theta(h)-\Theta(-h) \big) \, \e^{\frac{h}{2} (A_1+\cdots+A_N) },
\end{equation*}
and the unitarity of both $\e^{\frac{h}{2} (A_1+\cdots+A_N) }$ and $\Phi(h)$.
\end{proof}
Clearly, the scheme is of order of accuracy $p$ if and only if 
\begin{equation} \label{order_p}
\mathcal{E}(h) := \Theta(h)-\Theta(-h) =  \mathcal{O}(h^{p+1}) \quad \mbox{as} \quad h \to 0. 
\end{equation}
Observe that, $\mathcal{E}(-h)=-\mathcal{E}(h)$, which shows  that symmetric schemes necessarily attain even orders of accuracy.

\begin{theorem} \label{th_compo}
Suppose that the time-symmetric $2s$-stage composition \eqref{palindromic_alpha} 
\[
   \Psi(h) = \chi(\alpha_{1} h) \, \chi^*(\alpha_{2} h) \cdots \chi(\alpha_2 h) \, \chi^*(\alpha_1 h), 
\]
with $(\alpha_1, \alpha_2, \ldots, \alpha_s) \in \mathbb{R}^s$ and
\[
  \chi(h) = \e^{h A_1} \, \e^{h A_2} \cdots \e^{h A_N}, \qquad \chi^*(h) = \e^{h A_N} \cdots \e^{h A_2} \, \e^{h A_1}, 
\]
 provides an approximation of even order $p$ to $\e^{h(A_1 + \cdots + A_N)}$,
with $A_j$ skew-adjoint operators verifying that $\|A_j\| = a_j < \infty$ for $j=1,\ldots, N$. Then, the following estimate holds:
\begin{equation*} 
  \left\|\Psi(h)-\e^{h \sum_{j=1}^N A_j} \right\| \le \frac{2^{-p}}{(p+1)!} \left( 1+ 2 \sum_{i=1}^s |\alpha_i| \right)^{p+1}  \, \left(h \sum_{j=1}^N a_j \right)^{p+1}.
\end{equation*}
\end{theorem}

\begin{proof}
It is straightforward to check that 
\[
\Theta(h)=\e^{h C_r} \cdots \e^{h C_1}, \qquad 
\Theta(-h) = \e^{-h C_r} \cdots \e^{-h C_1}
\]
 for certain bounded skew-adjoint operators $C_j$ such that 
\begin{equation*}
\|C_1\|+\cdots+\|C_r\|  = \left( \frac{1}{2} + |\alpha_1|+\cdots+|\alpha_s| \right) \, \sum_{j=1}^N a_j.
\end{equation*}
Since $\mathcal{E}(h) = \mathcal{O}(h^{p+1})$, Lemma~\ref{l:Theta} and Proposition~\ref{p:boundsB} lead to
\[
  \|\mathcal{E}(h)\| \le  \frac{2 h^{p+1}}{(p+1)!} \left( \|C_1\|+\cdots+\|C_r\| \right)^{p+1}
 \]
 whence the conclusion follows. 
\end{proof}
As before, this result extends naturally to compositions \eqref{eq:palindromic_gamma}:

\begin{corollary} \label{th_compo_sym}
Suppose that the sequence $(\gamma_1, \ldots, \gamma_s) \in \mathbb{R}^s$ is palindromic, i.e., it is such that \eqref{eq:palindromic_gamma} holds, and that the symmetric $s$-stage composition 
\[
  \Psi(h) = S_2(\gamma_s h) \,  \cdots \, S_2(\gamma_2 h) \, S_2(\gamma_1 h)
\]
where $S_2(h) = \chi(h/2) \chi^*(h/2)$ with
\[
  \chi(h) = \e^{h A_1} \, \e^{h A_2} \cdots \e^{h A_N}, \qquad \chi^*(h) = \e^{h A_N} \cdots \e^{h A_2} \, \e^{h A_1}, 
\]
 provides an approximation of even order $p$ to $\e^{h(A_1 + \cdots + A_N)}$,
with $A_j$ skew-adjoint operators verifying that $\|A_j\| = a_j < \infty$ for $j=1,\ldots, N$. Then, the following estimate holds:
\begin{equation*} 
  \left\|\Psi(h)-\e^{h \sum_{j=1}^N A_j} \right\| \le \frac{2^{-p}}{(p+1)!} \left( 1+ \sum_{i=1}^s |\gamma_i| \right)^{p+1}  \, \left(h \sum_{j=1}^N a_j \right)^{p+1}.
\end{equation*}
\end{corollary}

By applying the same techniques, a similar result can be obtained for splitting schemes \eqref{gen.split} that are symmetric:
\begin{theorem} 
\label{th_gen.split_sym}
Suppose that the product formula \eqref{gen.split} admits a factorization of the form \eqref{eq:sym_factorization}
and provides an approximation of even order $p$ to $\e^{h(A_1 + \cdots + A_N)}$,
with $A_j$ skew-adjoint operators verifying that $\|A_j\| =  a_j < \infty$ for $j=1,\ldots, N$. Then, the following estimate holds:
\begin{equation*}
  \left\|\Psi(h)-\e^{h \sum_{j=1}^N A_j} \right\| \leq \frac{h^{p+1}}{2^p (p+1)!} \left[ \sum_{i=1}^N a_i \, \left(1+|\alpha_i^{(1)}|+\cdots+|\alpha_i^{(s)}|\right)\right]^{p+1}.
  \end{equation*}
\end{theorem}
As is well known, in the case of just two operators, $A = A_1 + A_2$, there exists a  close connection between composition \eqref{eq:Psi_alpha}  
and the splitting method
\begin{equation} \label{splitting_2}
  \Psi(h) = \e^{c_{s+1} t A_1}  \, \e^{d_s t A_2} \, \e^{c_s t A_1} \, \ldots \, \e^{d_1 t A_2} \, \e^{c_1 t A_1}, 
\end{equation}
with $c_j, d_j \in \mathbb{R}$. Specifically, \eqref{eq:Psi_alpha} can be rewritten as \eqref{splitting_2} if 
\begin{equation} \label{comp_split_con}
c_1 = \alpha_1, \quad \mbox{ and } \qquad c_{j+1} = \alpha_{2j} + \alpha_{2j+1}, \qquad d_{j} = \alpha_{2j-1} + \alpha_{2j}, \qquad j=1,\ldots,s
\end{equation}
(with $\alpha_{2s+1} = 0$). Conversely, any integrator \eqref{splitting_2} satisfying the condition $\sum_{j=1}^{s+1} c_j = \sum_{j=1}^s d_j$ can be expressed 
in the form \eqref{eq:Psi_alpha} \cite{mclachlan95otn}. In consequence, 
\begin{theorem} \label{th_split_2}
Given the splitting scheme
\[
  \Psi(h) = \e^{c_{s+1} t A_1}  \, \e^{d_s t A_2} \, \e^{c_s t A_1} \, \ldots \, \e^{d_1 t A_2} \, \e^{c_1 t A_1}, 
\]  
with $c_j, d_j \in \mathbb{R}$, providing an approximation of order $p$ to $\e^{h(A_1+A_2)}$, with $A_j$ skew-adjoint operators such that 
$\|A_j\| = a_j < \infty$, $j=1,2$, then
\[
  \| \Psi(h)  - \e^{h(A_1+A_2)} \| \le \frac{h^{p+1}}{(p+1)!} \left[ 1 +  \left( \sum_{i=1}^{2s} |\alpha_i| \right)^{p+1} \right] \, (a_1 + a_2)^{p+1},
\] 
where the coefficients $\alpha_i$ are related with $c_j, d_j$ through \eqref{comp_split_con}. If in addition, $\Psi(h)$ is time-symmetric, i.e., 
\[
  c_{s-j+2} = c_j, \qquad d_{s-j+1} = d_j, \qquad \mbox{ for all } j,
\]
then $\alpha_{2s-j+1} = \alpha_j$, the order $p$ is even and   
\[
  \| \Psi(h)  - \e^{h(A_1+A_2)} \| \le \frac{h^{p+1}}{2^p (p+1)!} \left[ 1 +  2 \left( \sum_{i=1}^{s} |\alpha_i| \right)^{p+1} \right] \, (a_1 + a_2)^{p+1},
\] 
\end{theorem}

\subsection{Global error}
Up to this point, we have examined the error produced after a single integration step. However, for practical applications it is also important 
to understand how this error accumulates over multiple steps, say after $k$ iterations.
The resulting global error can be estimated as follows. We have
\[
  \|u_k - u(t_k)\| = \| \Psi^k(h) u_0 - \e^{k h A} u_0\|,
\]
but we can write
\[
\Psi^k(h) u_0 - \e^{k h A} u_0 = \sum_{j=0}^{k-1} \Psi^{k-1-j}(h) \, \big(\Psi(h) - \e^{h A} \big) \, \e^{j h A} u_0.
\]
Since both $\Psi(h)$ and $\e^{h A}$ are unitary (so that $\|\Psi^m(h) v\| = \|v\|$ for any vector $v \in \mathcal{H}$), we obtain
\[
\begin{aligned}
\| \Psi^k(h) u_0- \e^{k h A} u_0\| & \leq \sum_{j=0}^{k-1} \|\big(\Psi(h) - \e^{h A}\big) \, \e^{j h A} u_0\| 
\leq \sum_{j=0}^{k-1} \|\Psi(h) - \e^{h A}\| \, \|\e^{j h A} u_0\| \\
& = k \, \| \Psi(h) - \e^{h A} \| \|u_0\|,
\end{aligned}
\]
so that, in the unitary case, the global error after $k$ steps can be bounded by $k$ times the local error.

\

Given two composition schemes, both of order $p$ but possibly with different numbers $s$ of stages, it is customary to compare their theoretical efficiency using the concept of {\em effective error} \cite{mclachlan95otn,blanes99siw,mclachlan02foh}. For the particular class of methods of the form \eqref{eq:comp_Strang}, it can be defined as
\begin{equation} \label{valueC}
   E_f  \equiv s \, \mathcal{C}^{1/p}, \qquad \text{where} \quad \mathcal{C} = \frac{1}{(p+1)!} \left[ 1 +  \left( \sum_{i=1}^s |\gamma_i| \right)^{p+1} \right].
\end{equation}

Indeed, suppose we apply a method of order $p$ with $s$ stages to approximate the exponential $\e^{A_1+\cdots+A_N}$ using $n$ Strang-type compositions, i.e., $\e^{A_1+\cdots+A_N} \approx \Psi(h)^k$ with $k = n/s$ and $h = s/n$. According to Theorem~\ref{th_compo},
\begin{equation}
\label{eq:global_error_estimate}
\|\e^{A_1+\cdots+A_N} - \Psi(h)^k\| \leq \frac{(a_1 + \cdots + a_N)^{p+1}}{n^p}\, \mathcal{C}\, s^p.
\end{equation}

For methods of the same order $p$, comparison amounts to evaluating the factor $s^p \mathcal{C}$ or, equivalently, the effective error $E_f = s \, \mathcal{C}^{1/p}$. The smaller the value of $E_f$, the smaller the global error bound for a fixed number $n$ of compositions.

For composition schemes of the form \eqref{eq:comp_Strang} that satisfy the symmetry condition \eqref{eq:palindromic_gamma}, 
Theorem~\ref{th_compo_sym} shows that a sharper estimate of the global error \eqref{eq:global_error_estimate} can be obtained by replacing $\mathcal{C}$ with the smaller constant
\begin{equation} \label{valueC_sym}
 \mathcal{C}_{\mathrm{sym}} = \frac{2^{-p}}{(p+1)!} \left(1 + \sum_{i=1}^s |\gamma_i| \right)^{p+1}.
\end{equation}
It then makes sense to redefine the effective error as
\begin{equation}
\label{eq:Ef_sym}
E_f  \equiv s \, \mathcal{C}_{\mathrm{sym}}^{1/p}.
\end{equation}


Table~\ref{tableSS} presents the values of $E_f$ for several of the most relevant schemes of this class available in the literature. For each order $p$, the methods are distributed according to their number of stages $\mathbf{s}$. In particular, the two fourth-order schemes collected in Table~\ref{tableSS} correspond to the well-known triple jump ($\mathbf{s} = 3$),
\begin{equation} \label{triple}
   S_2(\gamma h) \, S_2((1 - 2\gamma) h) \, S_2(\gamma h), \qquad \text{with} \quad \gamma = \frac{1}{2 - 2^{1/3}},
\end{equation}
and the quintuple jump composition ($\mathbf{s} = 5$),
\begin{equation} \label{quintuple}
S_2(\gamma h) \, S_2(\gamma h) \, S_2((1 - 4\gamma) h) \, S_2(\gamma h) \, S_2(\gamma h), \qquad \text{with} \quad 
\gamma = \frac{1}{4 - 4^{1/3}}.
\end{equation}
The boxed numbers in Table~\ref{tableSS} indicate the recommended schemes from \cite{blanes24smf}, based on the numerical experiments presented therein. For more details and references, we refer the reader to \cite[Table 8.1]{blanes24smf}. We can see the strong correlation between the theoretical estimates obtained
here and the performance of the methods in practice.

\begin{table}
  \centering
\begin{tabular}{|l|l|l|l|} \hline
\multicolumn{4}{|c|}{\bfseries  $\Psi(h) = S_2(\gamma_s h) \, S_2(\gamma_{s-1} h) \, \cdots \, S_2(\gamma_2 h) \, S_2(\gamma_1 h)$}
\\ \hline
$p=4$ & $p=6$ & $p=8$ & $p=10$ \\ \hline
\textbf{3}--$\; 2.892$ & \textbf{7}--$\; 6.454$ & {\bf 15}--$\; 12.489$ & {\bf 31}--$\; 24.867$ \\
\framebox[3mm]{\bf 5}--$\; 2.166$ & {\bf 9}--$\; 5.199$ & {\bf 17}--$\; 11.756$ & \framebox[4.5mm]{\bf 35}--$\; 22.548$ \\
      & \framebox[4.5mm]{\bf 13}--$\; 4.361$ & \framebox[4.5mm]{\bf 21}--$\; 9.170$ & \\ \hline
\end{tabular}
  \caption{Value of the effective error $E_f = s \, \mathcal{C}_{\mathrm{sym}}^{1/p}$, where $\mathcal{C}_{\mathrm{sym}}$ is defined in \eqref{valueC_sym}, for each symmetric composition $\Psi(h)$ of order $p = 4, 6, 8, 10$ with $\mathbf{s}$ stages. The recommended methods, according to \cite{blanes24smf}, are boxed. In all cases, they all possess the smallest value of $E_f$.}
\label{tableSS}
\end{table}

\section{Improved bounds based on the norm of the operators}

The local error bounds obtained in section \ref{sec.2} are essentially based on Proposition~\ref{p:boundsB}  with $q=p$, where $p$ is the order of accuracy of the method. It turns out, however, that improved bounds can be obtained by estimating 
explicitly the contributions arising from the terms of order $p+1$ through $q$, for a given $q \geq p$, and then applying again Proposition~\ref{p:boundsB}. This will lead to estimates of the form
\[
  \left\|\Psi(h)-\e^{h \sum_{j=1}^N A_j} \right\| \le \sum_{n=p+1}^{q} d_n \, \lambda^n +  \mathcal{C}_{q+1} \, \lambda^{q+1},
\]
where $\lambda \equiv h (a_1 + a_2 + \cdots + a_N)$ and $d_{p+1}, \ldots, d_{q}, \mathcal{C}_{q+1}$ depending on the coefficients of the particular scheme.

 Whereas $C_{q+1}$ has been already obtained in the previous section for different classes of schemes, we now determine explicitly the
coefficients $d_n$ for compositions of the basic Lie--Trotter scheme and the Strang splitting. As before, 
we analyze two different settings: first, general compositions, and then methods with a palindromic sequence of coefficients.

\subsection{Series expansions with multi-indices}
To begin with, we express the basic map~\eqref{eq:chi} as a power series in $h$:
\[
\chi(h) = I + h X_1 + h^2 X_2 + h^3 X_3 + \cdots,
\]
where
\[
X_n = \sum_{r_1 + \cdots + r_N = n} \frac{1}{r_1! \cdots r_N!} A_1^{r_1} \cdots A_N^{r_N}, \qquad n \ge 1.
\]
The first two terms are
\[
X_1 = A = 
A_1 + \cdots + A_N, \qquad
X_2 = \sum_{i=1}^N \left( \frac{1}{2} A_i^2 + A_i \sum_{j=i+1}^{N} A_j \right).
\]

As shown in~\cite{blanes24smf}, the composition $\Psi(h)$ defined in~\eqref{eq:Psi_alpha} admits a series expansion of the form
\begin{equation} \label{expan_Psi}
 \Psi(h) = I + \sum_{n \ge 1} h^n \sum_{m \ge 1} \sum_{i_1 + \cdots + i_m = n} w_{i_1,\ldots,i_m}(\alpha_1,\ldots,\alpha_{2s}) \,  X_{i_m} \cdots X_{i_1},
\end{equation}
where the coefficient functions $w_{i_1,\ldots,i_m}$ are recursively defined as
\begin{equation}
  \label{eq:wfun}
  \begin{split}
w_{i}(\alpha_1,\ldots,\alpha_{2\ell}) &=
   \sum_{j = 1}^s \left(\alpha_{2j}^{i} - (-\alpha_{2j-1})^{i} \right), \\
w_{i_1,\ldots,i_m}(\alpha_1,\ldots,\alpha_{2\ell}) &=
    \sum_{j = 1}^s \left(\alpha_{2j}^{i_m} - (-\alpha_{2j-1})^{i_m} \right) w_{i_1,\ldots,i_{m-1}}(\alpha_1,\ldots,\alpha_{2j-1}),\\
w_{i_1,\ldots,i_m}(\alpha_1,\ldots,\alpha_{2\ell-1}) &= w_{i_1,\ldots,i_m}(\alpha_1,\ldots,\alpha_{2\ell-1},0).
\end{split}
\end{equation}

Notice that the exact solution of eq. \eqref{basic_eq} can also be written in the form \eqref{expan_Psi}, namely
\begin{equation} \label{exact_sol}
\begin{split}
\exp(h A) &= I + \sum_{n \ge 1} \frac{h^n}{n!} A^n \\
&= I + \sum_{n \ge 1} h^n \sum_{m \ge 1} \sum_{i_1 + \cdots + i_m = n} \epsilon_{i_1,\ldots,i_m} \,  X_{i_m} \cdots X_{i_1},
\end{split}
\end{equation}
with
\begin{equation}
\label{eq:epsilon}
\epsilon_{i_1,\ldots,i_m} =
\begin{cases}
\frac{1}{m!} & \text{if } (i_1,\ldots,i_m) = (1, \ldots, 1), \\
0 & \text{otherwise}.
\end{cases}
\end{equation}

Subtracting~\eqref{exact_sol} from \eqref{expan_Psi} gives
\begin{equation} \label{diff_compo}
\Psi(h) - \exp(h A) = \sum_{n \ge 1} h^n \sum_{m \ge 1} \sum_{i_1 + \cdots + i_m = n} \delta_{i_1,\ldots,i_m}(\alpha_1,\ldots,\alpha_{2s}) \,  X_{i_m} \cdots X_{i_1},
\end{equation}
where
\begin{equation}
\label{eq:delta}
\delta_{i_1,\ldots,i_m}(\alpha_1,\ldots,\alpha_{2s}) = w_{i_1,\ldots,i_m}(\alpha_1,\ldots,\alpha_{2s}) - \epsilon_{i_1,\ldots,i_m}.
\end{equation}
Thus, the product formula~\eqref{eq:Psi_alpha} is a method of order $p$ if and only if
\begin{equation}
\label{eq:ocond}
\delta_{i_1,\ldots,i_m}(\alpha_1,\ldots,\alpha_{2s}) = 0
\end{equation}
for all multi-indices $(i_1,\ldots,i_m)$ satisfying $1 \le i_1 + \cdots + i_m \le p$.

Of course, in the case of the composition of Strang splittings \eqref{eq:comp_Strang}, the corresponding expansion is
\begin{equation}
\label{eq:error_expansion}
\Psi(h) - \exp(h A) = \sum_{n \ge 1} h^n \sum_{m \ge 1} \sum_{i_1 + \cdots + i_m = n} \hat{\delta}_{i_1,\ldots,i_m}(\gamma_1,\ldots,\gamma_s) \,  X_{i_m} \cdots X_{i_1},
\end{equation}
with
\begin{equation}
\label{eq:hdelta}
\hat{\delta}_{i_1,\ldots,i_m}(\gamma_1,\ldots,\gamma_s) = 
\delta_{i_1,\ldots,i_m} \left(\frac{\gamma_1}{2},\frac{\gamma_1}{2},\frac{\gamma_2}{2},\frac{\gamma_2}{2},\ldots,\frac{\gamma_s}{2}, \frac{\gamma_s}{2} \right).
\end{equation}

\subsection{Refined local error estimates for general sequences of coefficients}

We now focus on bounding the term of order $n$ in the local error expansion \eqref{diff_compo} for $n = p+1, \ldots, q$. Under the assumption that $\|A_j\| = a_j < \infty$ for $j = 1, \ldots, N$, it is straightforward to prove by induction that
\begin{equation}
\label{eq:boundXn}
\|X_n\| \leq \frac{(a_1 + \cdots + a_N)^n}{n!}, \qquad n \geq 1.
\end{equation}
Therefore, the contribution of order $n$ in the expansion \eqref{diff_compo} satisfies
\[
\left\|
h^n \sum_{m \geq 1} \sum_{i_1 + \cdots + i_m = n}  \delta_{i_1,\ldots,i_m}(\alpha_1,\ldots,\alpha_{2s}) \,  X_{i_m} \cdots X_{i_1}
\right\| \leq 
d_n(\alpha_1,\ldots,\alpha_{2s}) \, \left(h \sum_{j=1}^N a_j \right)^n,
\]
where for each $n \geq 1$, we define
\begin{equation}
\label{eq:dn}
d_n(\alpha_1,\ldots,\alpha_{2s}) = \sum_{m \geq 1} \sum_{i_1 + \cdots + i_m = n} \frac{|\delta_{i_1,\ldots,i_m}(\alpha_1,\ldots,\alpha_{2s})|}{i_1! \cdots i_m!}.
\end{equation}
In consequence, we have the following refined estimate.

\begin{theorem} \label{th_compo_refined_G}
Let
\[
   \Psi(h) = \chi(\alpha_{1} h) \, \chi^*(\alpha_{2} h) \cdots \chi(\alpha_2 h) \, \chi^*(\alpha_1 h), 
\]
where $(\alpha_1, \alpha_2, \ldots, \alpha_s) \in \mathbb{R}^s$ and
\[
  \chi(h) = \e^{h A_1} \, \e^{h A_2} \cdots \e^{h A_N}, \qquad \chi^*(h) = \e^{h A_N} \cdots \e^{h A_2} \, \e^{h A_1}, 
\]
be a $2s$-stage composition of order $p$ approximating $\e^{h(A_1 + \cdots + A_N)}$, where each $A_j$ is skew-adjoint and $\|A_j\| = a_j  <\infty$ for $j = 1,\ldots, N$. Then, for $q> p$, the local error satisfies
\[
\left\| \Psi(h) - \e^{h \sum_{j=1}^N A_j} \right\| \leq
\sum_{n = p+1}^{q} d_n(\alpha_1,\ldots,\alpha_{2s}) \, \lambda^n +
\frac{1}{(q+1)!} \left[ 1 + 2 \left( \sum_{i=1}^{2s} |\alpha_i| \right)^{q+1} \, \right] \lambda^{q+1},
\]
where $\lambda = h (a_1 + \cdots + a_N)$ and $d_n(\alpha_1,\ldots,\alpha_{2s})$ is given by \eqref{eq:dn}. 
\end{theorem}
If a composition of the Strang splitting is considered instead, the corresponding estimate is the following:

\begin{corollary} \label{th_compo_refined_S}
Let
\[
\Psi(h) = S_2(\gamma_s h) \cdots S_2(\gamma_2 h) S_2(\gamma_1 h),
\]
where $(\gamma_1, \ldots, \gamma_s) \in \mathbb{R}^s$ and $S_2(h) = \chi(h/2)\chi^*(h/2)$ with
\[
\chi(h) = \mathrm{e}^{h A_1} \cdots \mathrm{e}^{h A_N}, \qquad 
\chi^*(h) = \mathrm{e}^{h A_N} \cdots \mathrm{e}^{h A_1},
\]
be an $s$-stage composition of order $p$ approximating $\mathrm{e}^{h(A_1 + \cdots + A_N)}$, where each $A_j$ is skew-adjoint and $\|A_j\| = a_j  <\infty$ for $j = 1,\ldots, N$. Then, for $q \geq p$, the local error satisfies
\[
\left\| \Psi(h) - \mathrm{e}^{h \sum_{j=1}^N A_j} \right\| \leq
\sum_{n = p+1}^{q} d_n(\gamma_1,\ldots,\gamma_s) \, \lambda^n +
\frac{1}{(q+1)!} \left[ 1 + \left( \sum_{i=1}^s |\gamma_i| \right)^{q+1} \, \right] \lambda^{q+1},
\]
where $\lambda = h (a_1 + \cdots + a_N)$ and
\[
  d_n(\gamma_1,\ldots,\gamma_{s}) = \sum_{m \geq 1} \sum_{i_1 + \cdots + i_m = n} \frac{|\hat{\delta}_{i_1,\ldots,i_m}(\gamma_1,\ldots,\gamma_{ss})|}{i_1! \cdots i_m!}.
\]
\end{corollary}

\subsection{Refined local error estimates for palindromic sequences of coefficients}

Recall that composition schemes of the form \eqref{palindromic_alpha} admit the factorization \eqref{eq:sym_factorization}--\eqref{eq:Phi_s_even} and the local error can be bounded according to Lemma~\ref{l:Theta} as
\[
\| \Psi(h) - \mathrm{e}^{h (A_1 + \cdots + A_N)} \| \leq \|\Theta(h)-\Theta(-h)\|,
\]
where $\Theta(h)=\Phi(h) \, \e^{-\frac{h}{2} (A_1 + \cdots + A_N)}$. Expanding $\Phi(h)$ yields
\[
\Phi(h) = I + \sum_{n \geq 1} h^n \sum_{m \geq 1} \sum_{i_1 + \cdots + i_m = n} w_{i_1, \ldots, i_m}(\alpha_1, \ldots, \alpha_s) \, X_{i_m} \cdots X_{i_1},
\]
and similarly,
\begin{align*}
\exp\left(-\tfrac{h}{2} A\right) &= I + \sum_{n \geq 1} (-1)^n \frac{h^n}{2^n n!} A^n \\
&= I + \sum_{n \geq 1} h^n \sum_{m \geq 1} \sum_{i_1 + \cdots + i_m = n} \left(-\tfrac{1}{2}\right)^{i_1 + \cdots + i_m} \epsilon_{i_1, \ldots, i_m} \, X_{i_m} \cdots X_{i_1}.
\end{align*}
Therefore,
\begin{align*}
\Theta(h) &= \Phi(h)\, \exp\left(-\tfrac{h}{2} A\right) \\
&= I + \sum_{n \geq 1} h^n \sum_{m \geq 1} \sum_{i_1 + \cdots + i_m = n} v_{i_1, \ldots, i_m}(\alpha_1, \ldots, \alpha_s) \, X_{i_m} \cdots X_{i_1},
\end{align*}
where
\begin{align*}
v_{i_1, \ldots, i_m}(\alpha_1, \ldots, \alpha_s) &= w_{i_1, \ldots, i_m}(\alpha_1, \ldots, \alpha_s) + \left(-\tfrac{1}{2}\right)^{i_1 + \cdots + i_m} \epsilon_{i_1, \ldots, i_m} \\
&\quad + \sum_{j=1}^{m-1} \left(-\tfrac{1}{2}\right)^{i_{1} + \cdots + i_j}\epsilon_{i_1, \ldots, i_j}  \, w_{i_{j+1}, \ldots, i_m}(\alpha_1, \ldots, \alpha_s),
\end{align*}

This yields the series expansion
\[
\Theta(h) - \Theta(-h) = 2 \sum_{n \geq 1} h^{2n-1} \sum_{m \geq 1} \sum_{i_1 + \cdots + i_m = 2n-1} v_{i_1, \ldots, i_m}(\alpha_1, \ldots, \alpha_s)\, X_{i_m} \cdots X_{i_1}.
\]
Bounding the remainder of order $m$ as in Theorem~\ref{th_compo_refined_G} and accounting for the odd powers in $h$ leads to the following results:

\begin{theorem} \label{th_compo_sym_refined_G}
Suppose that the time-symmetric $2s$-stage composition \eqref{palindromic_alpha} 
\[
   \Psi(h) = \chi(\alpha_{1} h) \, \chi^*(\alpha_{2} h) \cdots \chi(\alpha_2 h) \, \chi^*(\alpha_1 h), 
\]
with $(\alpha_1, \alpha_2, \ldots, \alpha_s) \in \mathbb{R}^s$ and
\[
  \chi(h) = \e^{h A_1} \, \e^{h A_2} \cdots \e^{h A_N}, \qquad \chi^*(h) = \e^{h A_N} \cdots \e^{h A_2} \, \e^{h A_1}, 
\]
 provides an approximation of even order $p$ to $\e^{h(A_1 + \cdots + A_N)}$,
with $A_j$ skew-adjoint operators verifying that $\|A_j\| = a_j < \infty$ for $j=1,\ldots, N$. Then, for $q\geq p$, the local error satisfies
\begin{equation*} 
  \left\|\Psi(h)-\e^{h \sum_{j=1}^N A_j} \right\| \le \sum_{n = p+1}^{q} d_n^{\mathrm{sym}}(\alpha_1,\ldots,\alpha_s)\, \lambda^n +
\frac{\lambda^{q+1}}{(q+1)! \, 2^{q}} \left(1 + 2 \sum_{i=1}^s |\alpha_i| \right)^{q+1},
\end{equation*}
where $\lambda = h (a_1 + \cdots + a_N)$ and
\begin{equation}
\label{eq:dn_sym}
d_n^{\mathrm{sym}}(\alpha_1,\ldots,\alpha_s) = 2 \sum_{m \geq 1} \sum_{i_1 + \cdots + i_m = n} \frac{|v_{i_1, \ldots, i_m}(\alpha_1, \ldots, \alpha_s)|}{i_1! \cdots i_m!}
\end{equation}
if $n$ is odd and $d_n^{\mathrm{sym}}(\alpha_1,\ldots,\alpha_s)=0$ if $n$ is even.
\end{theorem}

\begin{corollary} \label{th_compo_sym_refined}
Let $(\gamma_1, \ldots, \gamma_s) \in \mathbb{R}^s$ satisfy the symmetry condition \eqref{eq:palindromic_gamma}, and let
\[
\Psi(h) = S_2(\gamma_s h) \cdots S_2(\gamma_2 h) S_2(\gamma_1 h),
\]
where $S_2(h) = \chi(h/2)\chi^*(h/2)$ with
\[
\chi(h) = \mathrm{e}^{h A_1} \cdots \mathrm{e}^{h A_N}, \qquad 
\chi^*(h) = \mathrm{e}^{h A_N} \cdots \mathrm{e}^{h A_1},
\]
be a symmetric composition of even order $p$ approximating $\mathrm{e}^{h(A_1 + \cdots + A_N)}$, where each $A_j$ is skew-adjoint and $\|A_j\| \le a_j$. Then, for $q \geq p$, the local error satisfies
\[
\left\| \Psi(h) - \mathrm{e}^{h \sum_{j=1}^N A_j} \right\| \leq 
\sum_{n = p+1}^{q} d_n^{\mathrm{sym}}(\gamma_1,\ldots,\gamma_s)\, \lambda^n +
\frac{\lambda^{q+1}}{(q+1)! \, 2^{q}} \left(1 + \sum_{i=1}^s |\gamma_i| \right)^{q+1},
\]
where $\lambda = h (a_1 + \cdots + a_N)$ and
\begin{equation}
\label{eq:dn_sym2}
d_n^{\mathrm{sym}}(\gamma_1,\ldots,\gamma_s) = 2 \sum_{m \geq 1} \sum_{i_1 + \cdots + i_m = n} \frac{|v_{i_1, \ldots, i_m}(\gamma_1, \ldots, \gamma_s)|}{i_1! \cdots i_m!}
\end{equation}
if $n$ is odd and $d_n^{\mathrm{sym}}(\gamma_1,\ldots,\gamma_s)=0$ if $n$ is even.
\end{corollary}

\subsection{Numerical illustration of the refined estimates}
\label{sec.3.4}

To illustrate how the error estimates provided in this section actually improve the general bounds obtained in section \ref{sec.2}, we take one
particular method in Table \ref{tableSS}, namely the time-symmetric scheme \eqref{eq:palindromic_gamma} with $s=13$ and order $p=6$ 
of the form
\begin{equation} \label{sofro_spa}
  \Psi(h) = S_2(\gamma_1 h) \cdots S_2(\gamma_6 h) S_2(\gamma_7 h) S_2(\gamma_6 h) \cdots S_2(\gamma_1 h)
\end{equation}
proposed in
\cite{sofroniou05dos}. 

First of all, we take $q=7$, and compare the general estimates given by Corollaries \ref{th_compoS} and \ref{th_compo_sym}. Carrying out the calculation
with the corresponding coefficients, we get, respectively, 
   \[
   \begin{aligned}
       \left\| \Psi(h) - \mathrm{e}^{h \sum_{j=1}^N A_j} \right\| & \le 0.0912032792 \, \lambda^7 \qquad \mbox{(generic)} \\
     & \le 0.0162837163 \, \lambda^7 \qquad \mbox{(time-symmetric)}.
   \end{aligned}
   \]
We see then that the estimate obtained for time-symmetric schemes improves the general one by almost one order of magnitude. 
Then we apply Corollary \ref{th_compo_sym_refined} for time-symmetric compositions with $q=9$ and $q=11$, thus leading to
\[
   \left\| \Psi(h) - \mathrm{e}^{h \sum_{j=1}^N A_j} \right\|  \le P_q(\lambda),
\]   
 with 
 \[
 \begin{aligned}
   & P_9(\lambda) =   10^{-4} \left( 0.287348 \, \lambda^7 + 6.536152 \, \lambda^9 \right) \\
   & P_{11}(\lambda) = 10^{-5} \left( 2.87348 \, \lambda^7 + 0.390482 \, \lambda^9 + 
 1.71724 \, \lambda^{11} \right)
 \end{aligned}
 \]
These curves are depicted in Figure \ref{fig_qs}. As we can see,  we are able to improve the estimate provided by Corollary
\ref{th_compo_sym} in more than 3 orders of magnitude for a wide range of values of the normalized time-step $\lambda$. This same pattern 
is actually observed for all the methods collected in Table \ref{tableSS}.

\begin{figure}[htb]
\centering
\includegraphics[scale=1]{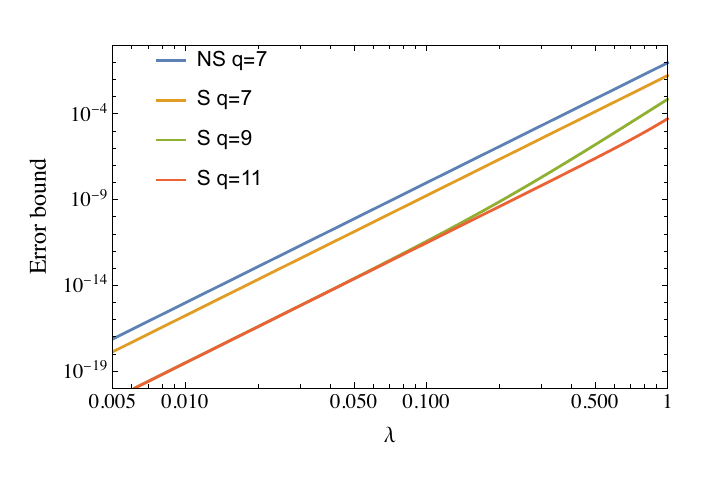} \\
\caption{\small{Error bounds for the time-symmetric 6th-order scheme \eqref{sofro_spa} with $s=13$. Lines correspond (from top to bottom) to the result provided by Corollary \ref{th_compoS} for a general composition of Strang splittings with $q=7$ (NS $q=7$), 
Corollary \ref{th_compo_sym} also with $q=7$ (S $q=7$), and the refined estimates provided by Corollary \ref{th_compo_sym_refined} for symmetric schemes with 
$q=9$ and $q=11$.}}
\label{fig_qs}
\end{figure}

\textbf{ALT TEX}: Graphical representation of different error bounds obtained in this work for a time-symmetric 6th-order composition method with 13 stages.

\section{Error bounds in terms of commutators}

While the previous method relies on computing power series expansions of both the exact solution and the numerical map, and then bounding the norm of their difference, the approach we now introduce is more indirect. Instead, we first derive the linear differential equation satisfied by the numerical flow. The error is then estimated by bounding the norm of the difference between the coefficient matrix of this equation and the operator $A = A_1 + \cdots + A_N$. 
In other words, rather than working with expansions in the unitary group, we consider expansions in the associated Lie algebra of skew-adjoint operators. This allows the resulting bounds to be expressed directly in terms of commutators of the operators involved. We now summarize the main steps of this approach.


\subsection{Estimates based on the modified equation satisfied by the splitting method}

As in Proposition~\ref{p:boundsB} and Corollary~\ref{c:expB_le}, we begin by considering the approximation of  $\e^{h (C_1+ \cdots + C_r)}$, where $C_1,\ldots,C_r$ are arbitrary skew-adjoint matrices, by the product of exponentials
\[
\e^{h C_r}\cdots \e^{h C_1}.
\]
Our goal is to estimate the norm of the difference
$\e^{h C_r}\cdots \e^{h C_1} - \e^{h (C_1+ \cdots + C_r)}$. We define 
\begin{equation}
\label{eq:M(t)}
M(t) = \left(\frac{d}{dt} \e^{t C_r}\cdots \e^{t C_1} \right) \e^{-t C_1}\cdots \e^{-t C_r}, 
\end{equation}
so that $\Psi(t)=\e^{t C_r}\cdots \e^{t C_1}$ is the solution operator of the initial value problem
\begin{equation} \label{modif1}
  \frac{d }{dt} \Psi(t) = M(t) \, \Psi(t), \qquad \Psi(0) = I.
\end{equation}

\begin{lemma} \label{lemma_estimate}
Given $r>1$ and arbitrary skew-adjoint matrices $C_1,\ldots,C_r$, then
\[
\|\e^{h C_r}\cdots \e^{h C_1} - \e^{h (C_1+ \cdots + C_r)}\| \leq 
\int_0^h \| M(t) - (C_1+\cdots+C_r)\| \, dt,
\]
where $M(t)$ is given by \eqref{eq:M(t)}.
\end{lemma}  

\begin{proof}
The operator $W(t)=\Psi(t) - \e^{t (C_1+ \cdots + C_r)}$ clearly verifies
\[
  \frac{d }{dt} W(t) = (C_1 + \cdots + C_r) W(t) - R(t) \Psi(t), \qquad W(0) = 0,
\]
where $R(t) = M(t) - (C_1+\cdots + C_r)$. Integrating this equation leads to
\[
  W(h) = - \int_0^h \e^{(h - t)(C_1 + \cdots + C_r)} R(t) \Psi(t) dt,  
\]
so that
\[
  \|W(h)\| \leq \int_0^h \|\e^{(h - t)(C_1 + \cdots + C_r)} \| \, \|R(t)\| \, \|\Psi(t)\| dt = \int_0^h  \|R(t)\|  \, dt, 
\]
since both $\e^{t (C_1+ \cdots + C_r)}$ and $\Psi(t)$ are unitary operators.
\end{proof}  
An explicit expression of \eqref{eq:M(t)} can be obtained  as 
\begin{equation} \label{mt_exp1}
\begin{aligned}
  M(t) & = C_r + \e^{t \, \ad_{C_r}} C_{r-1} +  \e^{t \, \ad_{C_r}} \e^{t \, \ad_{C_{r-1}}} C_{r-2} + \cdots \\
    & \quad + \e^{t \, \ad_{C_r}} \cdots \, \e^{t \, \ad_{C_3}} C_2 + \e^{t \, \ad_{C_r}} \cdots \, \e^{t \, \ad_{C_2}} C_1,
\end{aligned}
\end{equation}    
where $\ad_X Y = [X,Y] = X Y  - Y X$, and
\[
  \e^{t \, \ad_{X}} Y = \e^{t \,  X} \, Y \, \e^{-t \,  X}.
  \]
 Observe that, if $X$ is a skew-adjoint operator, then
 \[
 \left\| 
  \e^{t \, \ad_{X}} Y
  \right\| 
  \leq \|Y\|.
 \]
It is straightforward to check that $S(t) = \e^{t \, \ad_{X}}$ actually satisfies 
\begin{equation*}
\frac{d}{dt} S(t) Y = S(t) \, \ad_{X} Y, \qquad S(0) = Y, 
\end{equation*}
which implies the well-known identity
\begin{equation}
\label{eq:exp_ad_infty}
\e^{t \, \ad_{X}} Y = \sum_{k \ge 0} \frac{t^k }{k!} \ad_X^k Y = \sum_{k \ge 0} \frac{t^k }{k!}
 \underbrace{[X, [X, \ldots, [X}_{k}, Y]]].
\end{equation}
%
 The expression of $M(t)$ given by \eqref{mt_exp1} can be obtained in a more compact way if we introduce 
the operators $Z_{k}$, $k=1,\ldots, r$, defined recursively as
\begin{equation} \label{zetas}
\begin{aligned}
  & Z_1  = C_1 \\
  & Z_k(t)  = C_k + \e^{t \, \ad_{C_k}}  Z_{k-1}(t), \qquad\qquad k=2,\ldots, r.
  \end{aligned}
\end{equation}
Then, clearly, $M(t) = Z_{r}(t)$. 

For each $k\geq 2$, $Z_k(t)$ can be expanded as a series in powers of $t$ as
\begin{equation} \label{zetas_b}
Z_k(t) = \sum_{n=0}^{\infty} t^{n} Z_{k,n}, 
\end{equation}
where
\begin{equation}
\label{eq:Zkn}
\begin{aligned}
Z_{k,0} &= C_1 + \cdots + C_k,\\
 Z_{k,n} &= \sum_{j=0}^n \frac{1}{j!} \,\ad_{C_k}^j Z_{k-1, n-j},\quad n\geq 1,
 \end{aligned}
\end{equation}
or explicitly, 
\begin{equation}
\label{eq:Zkn_explicit}
\begin{aligned}
 & Z_{1,n}=0, \qquad n \ge 1 \\ 
 & Z_{k,n} = \sum_{j=2}^{k} \, \sum_{i_j+\cdots+i_k=n} \frac{\ad_{C_k}^{i_k} \cdots \ad_{C_j}^{i_j}}{i_k! \cdots i_j!} C_{j-1}, \quad k\geq 2, \quad n\geq 0.
\end{aligned} 
\end{equation}
We next aim at expressing $M(t)$ as the $n$th order truncation of that series plus a remainder.
 To this end, we will apply Lemma~\ref{l:lemma2} below, which is a simple consequence of the Taylor expansion of the exponential function
with remainder in integral form. 
\begin{lemma} \label{l:lemma2}
    Let $X, Y$ be two linear operators and $t \in \mathbb R$. Then, for any $n\geq 1$, 
        \[
        \e^{t \, \ad_X} Y = \sum_{k=0}^{n-1} \frac{t^k}{k!} \, \ad_X^{k} Y + t^n G_n\left(t, X, Y\right),
    \]
    where
  \begin{equation}
  \label{eq:Gn}
G_n(t, X,Y) = \int_0^1 \frac{\sigma^{n-1}}{(n-1)!} \, \e^{t \, (1-\sigma)\, \ad_X}\, \ad_X^n Y\, d\sigma.
\end{equation}

\end{lemma}


We can express \eqref{zetas_b} for any given $q \ge 1$ as
\begin{equation}
\label{eq:Zk_with_remainder}
Z_k(t) = \sum_{n=0}^{q-1} t^{n} Z_{k,n} + t^q R_{k,q}(t), \qquad k\geq 2,
\end{equation}
where the expression of $R_{k,q}(t)$ can be obtained from \eqref{zetas} as follows:
\[
\begin{aligned}
Z_k(t) &= C_k + \e^{t \, \ad_{C_k}}  Z_{k-1}(t) = C_k + \sum_{n=0}^{q-1} t^{n} \e^{t \, \ad_{C_k}}  Z_{k-1,n} + t^q \e^{t \, \ad_{C_k}} R_{k-1,q}(t) \\
&=  C_k + \sum_{n=0}^{q-1} t^n \sum_{j=0}^{q-1-n} \frac{t^{j}}{j!} \ad_{C_k}^{j} Z_{k-1,n} + t^q \sum_{n=0}^{q-1}  G_{q-n}\left(t, C_k, Z_{k-1,n}\right) +
 t^q \e^{t \, \ad_{C_k}} R_{k-1,q}(t) \\
 & = \sum_{n=0}^{q-1} t^{n} Z_{k,n} + t^q \sum_{n=0}^{q-1}  G_{q-n}\left(t, C_k, Z_{k-1,n}\right) +
 t^q \e^{t \, \ad_{C_k}} R_{k-1,q}(t),\
 \end{aligned}
 \]
with $R_{1,q}(t)=0$. We thus have, for $k\geq 2$,
\begin{equation}
\label{eq:R_{k,q}}
R_{k,q}(t) = \hat R_{k,q}(t) + \e^{t \, \ad_{C_k}} R_{k-1,q}(t),
\end{equation}
where 
\begin{equation*}
\hat R_{k,q}(t) \equiv  \sum_{n=1}^{q}  G_{n}\left(t, C_k, Z_{k-1,q-n}\right).
\end{equation*}
In consequence,
\begin{equation}  \label{eq_M(t)}
M(t) = Z_r(t) = \sum_{j=1}^{r} C_j + \sum_{n=1}^{q-1} t^n Z_{r,n} + t^q \, R_{r,q}(t), \qquad q \ge 1,
\end{equation}
where 
\begin{equation*}
\| R_{r,q}(t) \| \leq \| \hat R_{2,q}(t) \| +  \cdots + \| \hat R_{r,q}(t) \|.
\end{equation*}
Taking into account the definition of $\hat R_{k,q}(t)$ and  \eqref{eq:Gn}, we can write
for $k\geq 2$
\begin{align*}
\hat R_{k,q}(t)
&= 
\int_0^1  \e^{t \, (1-\sigma)\, \ad_{C_k}}\, 
\left(  \sum_{n=1}^{q} \frac{\sigma^{n-1}}{(n-1)!} \, \ad_{C_k}^n Z_{k-1,q-n} 
\right)
d\sigma,
\end{align*}
and hence,
\begin{align*}
\left\| \hat R_{k,q}(t) \right\| &\leq
\int_0^1
\left\|
 \sum_{n=1}^{q} \frac{\sigma^{n-1}}{(n-1)!} \, \ad_{C_k}^n Z_{k-1,q-n} 
\right\|
d\sigma.
\end{align*}

\begin{lemma}
\label{l:M(t)}
Consider \eqref{mt_exp1} and \eqref{eq:Zkn} for arbitrary skew-adjoint matrices $C_1,\ldots,C_r$.  Then 
\begin{equation*}
\left\| M(t) - \sum_{n=0}^{q-1} t^n \, Z_{r,n}  \right\| \leq 
t^q \sum_{k=2}^r 
\int_0^1
\left\|
 \sum_{n=1}^{q} \frac{\sigma^{n-1}}{(n-1)!} \, \ad_{C_k}^n Z_{k-1,q-n} 
\right\|
d\sigma.
\end{equation*}
\end{lemma}

This, together with Lemma~\ref{lemma_estimate} leads to the following result:

\begin{proposition} \label{p:basic_comm}
Given arbitrary skew-adjoint matrices $C_1,\ldots,C_r$,
\[
\| \e^{h C_r}\cdots \e^{h C_1} - \e^{h (C_1 + \cdots + C_r)} \| \leq \sum_{n=1}^{q-1}  \frac{h^{n+1}}{n+1} \| Z_{r,n} \| + h^{q+1} \, \mathcal{R}_q(h),
\]
where 
\begin{equation}
\label{eq:R_q}
\begin{aligned}
\mathcal{R}_q(h) &= \frac{1}{q+1}
 \sum_{k=2}^r \int_0^1
\left\|
 \sum_{n=1}^{q} \frac{\sigma^{n-1}}{(n-1)!} \, \ad_{C_k}^n Z_{k-1,q-n} 
\right\|
d\sigma,
\end{aligned}
\end{equation}
and the skew-adjoint operators $Z_{k,n}$ are given by the recursion \eqref{eq:Zkn}.
\end{proposition}

At this point it is illustrative to specify these bounds for the most elementary splitting methods.

\paragraph{(a) Lie--Trotter.} We consider $\Psi(h) = \e^{h A_1} \cdots \e^{h A_N}$ for skew-adjoint operators $A_1,\ldots,A_N$. Application of Proposition~\ref{p:basic_comm} with $r=N$, $C_j = A_j$, $q=1$, directly gives
\[
  \| \e^{h A_1} \cdots \e^{h A_N} - \e^{h (A_1 +\cdots + A_N)} \| \leq  \frac{h^2}{2} \, \sum_{k=2}^{N}
 \left\|  \sum_{j=1}^{k-1}   [A_k,  A_j] \, \right\|.
   \] 
   
   \paragraph{(b) Strang splitting with two operators.}
We now consider the Strang splitting with $N=2$, that is $\Psi(t) = \e^{\frac{t}{2} A_1} \, \e^{t A_2} \, \e^{\frac{t}{2} A_1}$.
Application of Proposition~\ref{p:basic_comm} with $q=2$, $r=3$, $C_1=C_3= \frac12 A_1$, $C_2=A_2$, $Z_{2,2}=0$ (since it is of second order) gives
\begin{equation*} 
  \|\e^{h (A_1 + A_2)} - \e^{\frac{h}{2} A_1} \, \e^{h A_2} \, \e^{\frac{h}{2} A_1}\|  \leq  h^3 \mathcal{R}_2(h),
\end{equation*}  
where
\[
\begin{aligned}
\mathcal{R}_2(h) &= \frac13 \int_0^1 \sigma  \, \|\ad_{C_2}^2 C_1\| \, d\sigma + \frac{1}{3} \int_0^1
\left\| \ad_{C_3} \ad_{C_2} C_1 + 
\sigma \, \ad_{C_3}^2 (C_1 + C_2) \right\|
d\sigma \\
&= \frac{1}{6} \, \|\ad_{C_2}^2 C_1\| +  \frac{1}{12} \int_0^1 \left\| (1-\sigma) \,  [A_1, [A_2,A_1]] \right\| d\sigma \\
&= \frac{1}{12} \, \|[A_2,[A_2, A_1]\| +  \frac{1}{24} 
\left\|
 [A_1, [A_1,A_2]]
\right\|.
\end{aligned}
\]
We thus have
\begin{equation} \label{b_strang}
\begin{aligned}
 & \|\e^{h (A_1 + A_2)} - \e^{\frac{h}{2} A_1} \, \e^{h A_2} \, \e^{\frac{h}{2} A_1}\|  \leq \frac{h^3}{12} \|[A_2,[A_1,A_2]]\| + \frac{h^3}{24} \|[A_1,[A_1,A_2]]\|,
\end{aligned} 
\end{equation}  
i.e, we get the optimal bound, reproducing previous results \cite{suzuki85dfo,childs21tot,iserles24aea}. The bound \eqref{b_strang} is
optimal in the sense that, in virtue of the symmetric Baker--Campbell--Hausdorff formula \cite{blanes25aci}, one has 
\begin{equation*}
\lim_{h\to 0} h^{-2}\left(\e^{h (A_1 + A_2)} - \e^{\frac{h}{2} A_1} \, \e^{h A_2} \, \e^{\frac{h}{2} A_1}\right) = \frac{1}{12} [A_2,[A_1,A_2]] + \frac{1}{24} [A_1,[A_1,A_2]],
\end{equation*}
so that the coefficient $\frac{1}{12}$ or $\frac{1}{24}$ cannot be replaced by smaller ones in \eqref{b_strang}. Indeed, this would lead to a contradiction for operators $A_1$ and $A_2$  such that $[A_1,[A_1,A_2]]=0$ or $[A_1,[A_1,A_2]]=0$ respectively.

\begin{remark} \label{r:l:M(t)}
The estimate for the remainder of the $(q-1)$th order truncated expansion of $M(t)$ given in Lemma~\ref{l:M(t)} above can be simplified by applying the triangular inequality to the integrand in the estimate as follows:
\begin{equation*}
\left\| M(t) - \sum_{n=0}^{q-1} t^n \, Z_{r,n}  \right\| \leq 
t^q \sum_{k=2}^r \sum_{n=1}^{q} 
 \left\| \frac{1}{n!} \,
  \ad_{C_k}^n Z_{k-1,q-n} 
\right\|.
\end{equation*}
However, using that simplified upper bound of the remainder of $M(t)$ would lead to the following suboptimal estimate of the local error of Strang splitting 
\begin{equation} \label{b_strang_so}
\begin{aligned}
 & \|\e^{h (A_1 + A_2)} - \e^{\frac{h}{2} A_1} \, \e^{h A_2} \, \e^{\frac{h}{2} A_1}\|  \leq \frac{h^3}{12} \|[A_2,[A_1,A_2]]\| + \frac{h^3}{8} \|[A_1,[A_1,A_2]]\|,
\end{aligned} 
\end{equation}  
instead of the optimal one \eqref{b_strang}.
\end{remark}

\paragraph{(c) Strang splitting with an arbitrary number of operators.}
Although Proposition \ref{p:basic_comm} provides the optimal bound for the local error committed by the Strang splitting with two operators, this is not generally the
case when $N>2$. A better estimate is obtained by following the same approach as in subsection \ref{subsec_2.3}, i.e., by considering the symmetric
factorization \eqref{eq:sym_factorization},
with $\Phi(h)=\e^{h A_N} \cdots \e^{h A_1}$ (or any other permutation of the factors) and $\Phi^*(h)=\Phi(-h)^{-1}$. To proceed, we need the following
preliminary result.

\begin{lemma}
\label{l:Theta_M(t)}
Let $\Theta(t)$ be a one-parameter family of unitary operators  and consider the skew-adjoint operator 
\begin{equation}
\label{eq:l:Theta_M(t)}
M(t) = 
\left(
\frac{d}{dt} \Theta(t)
\right)
\Theta(t)^{-1},
\end{equation}
so that $\Theta(t)$ is the solution operator of 
\begin{equation*}
\frac{d}{dt} \Theta(t) = M(t) \Theta(t), \qquad  \Theta(0)=I.
\end{equation*}
Then, for all $h \in \mathbb{R}$,
\begin{equation*}
\| \Theta(h) - \Theta(-h)\| \leq \int_0^h \| M(t) + M(-t)\|\, dt.
\end{equation*}
\end{lemma}
\begin{proof}
The proof is very similar indeed to the one of Lemma \ref{lemma_estimate}. In this case we take $W(t) = \Theta(t) - \Theta(-t)$, so that 
\[
  \frac{d }{dt} W(t) = M(t) W(t) + B(t), \qquad W(0) = 0,
\]
with $B(t) \equiv (M(t) + M(-t)) \, \Theta(-t)$. In consequence \cite{coddington55tod},
\[
    W(t) = \Xi(t) \int_0^t \Xi^{-1}(s) B(s) \, ds,
\]
where $\Xi(t)$ is the fundamental matrix of the system $\frac{d }{dt} W(t) = M(t) W(t)$. Since $\| \Xi(t)\| = 1$, then
\[
  \|W(t)\| \le \int_0^t \|B(s)\| ds \le \int_0^t \| M(s) + M(-s)\| \, ds,
\]  
and the result follows.
\end{proof}

Equipped with this lemma, we can now deduce a refined error estimate for the 2nd-order scheme
\[
    S_2(h) = \e^{\frac{1}{2} h A_1} \,  \e^{\frac{1}{2} h A_2} \, \cdots \,  \e^{h A_N} \, \cdots \e^{\frac{1}{2} h A_2} \, \e^{\frac{1}{2} h A_1}.
\]
Notice that $S_2(h)$ admits the factorization \eqref{eq:sym_factorization}, $S_2(h) = \Phi^*(h) \Phi(h)$, with 
$\Phi(h)=\e^{\frac{h}{2} A_N} \cdots \e^{\frac{h}{2} A_1}$. Let us introduce the operator
\begin{equation} \label{eq_Theta}
   \Theta(t)=\Phi(t) \e^{-\frac{t}{2} (A_1+\cdots+A_N)} = \e^{\frac{t}{2} A_N} \cdots \e^{\frac{t}{2} A_1} \, \e^{-\frac{t}{2} (A_1+\cdots+A_N)}.
\end{equation}
A straightforward  calculations shows that
\[
  S_2(h) - \e^{h (A_1 + \cdots A_N)} = \Phi^*(h) \big( \Theta(h) - \Theta(-h) \big) \e^{\frac{h}{2} (A_1 + \cdots A_N)}
\]
and application of Lemma \ref{l:Theta_M(t)} leads to
\begin{equation} \label{es_S_1}
    \|S_2(h) - \e^{h (A_1 + \cdots + A_N)}\| \le \|\Theta(h) - \Theta(-h) \| \le \int_0^h \|M(t) + M(-t)\| \, dt,
\end{equation}
where $M(t)$ is given by \eqref{eq:l:Theta_M(t)} and $\Theta(t)$ by \eqref{eq_Theta}. In consequence, 
by taking $r=N+1$, $q=2$,  $C_1 = -\frac12 (A_1+\cdots+A_N)$, $C_{j+1} = \frac{1}{2}A_j$, $j=1,\ldots,N$ in
expression \eqref{eq_M(t)}, we have
\[
  M(t) = t Z_{r,1} + t^2 R_{r,2}(t),
\]
and Lemma \ref{l:Theta_M(t)} gives
\[
   \|M(t) + M(-t) \| \le 2 \, t^2 \sum_{k=2}^{N+1} \int_0^1 \| \ad_{C_k} Z_{k-1,1} + \sigma \, \ad_{C_k}^2 Z_{k-1,0} \| \, d\sigma.
\]
Since $Z_{k-1,0} =  -\frac{1}{2} (A_{k-1} + \cdots + A_N)$, we have
\[
\begin{aligned}
 & \ad_{C_k}^2 Z_{k-1,0} = -\frac{1}{8} \ad_{A_{k-1}}^2 (A_k + \cdots + A_N), \\
 & \ad_{C_k}  Z_{k-1,1} = \ad_{C_k} \sum_{j=1}^{k-2} \ad_{C_{j+1}} Z_{j,0} = -\frac{1}{8} \sum_{j=1}^{k-2} \ad_{A_{k-1}} \ad_{A_j} (A_{j+1} + \cdots + A_N)
\end{aligned} 
\]
and \eqref{es_S_1} leads finally to
\[
\begin{aligned}
   \|S_2(h) - \e^{h (A_1 + \cdots + A_N)}\| & \le  \frac{h^3}{12} \sum_{k=2}^{N+1} \Big( \big\| [A_{k-1}, \sum_{j=1}^{k-2} [A_j, A_{j+1} + \cdots + A_N]] \big\|   \\
   & \qquad\quad +
   \frac{1}{2} \big\| [A_{k-1}, [A_{k-1}, A_k + \cdots + A_N]] \big\| \Big).
\end{aligned}   
\]   
This expression reproduces the previous estimate \eqref{b_strang} when $N=2$.

\subsection{Error bounds for splitting methods with two operators}

\subsubsection{Generic estimates}

The case when the problem \eqref{basic_eq} can be separated into only two parts, i.e.,
\begin{equation} \label{equ_2op}
 \frac{du}{dt} = (A + B) u, \qquad u(0) = u_0,
\end{equation}
 is very common in applications. Here and in the sequel the two skew-adjoint operators involved are denoted as $A$ and $B$ for simplicity, and the goal
 is to get specific error estimates for time-symmetric splitting methods of the form
\begin{equation} \label{2oper_sym1}
  \Psi(h) = \e^{h a_1 A} \, \e^{h b_1 B} \, \cdots \, \e^{h a_s  A} \, \e^{h b_s B} \, \e^{h a_s A} \, \cdots \, \e^{h b_1 B} \, \e^{h a_1 A},
\end{equation}
in terms of a basis of the free Lie algebra generated by the two operators $A$ and $B$. Notice that since it is possible to set $b_s = 0$, this format includes 
schemes where the central exponential involves the operator $A$  instead of $B$. Likewise, integrators whose first (and last) term corresponds to the exponential of $B$  can be
reproduced by \eqref{2oper_sym1} by setting to zero the corresponding coefficients.

Proceeding in an analogous was as for the Strang splitting, we introduce the operators
\[
  \Phi(h) = \e^{\frac{h}{2}  b_s B} \, \e^{h a_s A} \, \cdots \, \e^{h b_1 B} \, \e^{h a_1 A},
\]
so that  $\Psi(h) = \Phi^*(h) \Phi(h)$, and
\begin{equation} \label{2oper_sym2}
  \Theta(t) = \Phi(t) \, \e^{-\frac{t}{2} (A+B)}.
\end{equation}
As before,
\[
  \Psi(h) - \e^{h(A+B)} = \Phi^*(h) \big( \Theta(h) - \Theta(-h) \big) \e^{\frac{h}{2} (A+B)}, 
\]
and therefore
\begin{equation} \label{2oper_sym3}
   \| \Psi(h) - \e^{h(A+B)} \| \le \|\Theta(h) - \Theta(-h) \| \le \int_0^h \|M(t) + M(-t)\| \, dt.
\end{equation}   
The expression of $M(t)$ can be obtained from \eqref{eq:l:Theta_M(t)} with   
\begin{equation} \label{theta_2}
 \Theta(t) = \e^{ \frac{t}{2} b_s B} \, \e^{t a_s A} \, \cdots \, \e^{t b_1  B} \, \e^{t a_1 A} \, \e^{-\frac{t}{2} (A+B)}.
\end{equation}
In other words, $M(t)$ is given by \eqref{eq_M(t)} with $r=2s+1$,
\begin{equation} \label{B_kop}
\begin{aligned}
  & C_1 = -\frac{1}{2} (A+B), \qquad C_{2s} = a_s A, \qquad C_{2s+1} = \frac{1}{2} b_s B \\
  & C_{2j} = a_j A, \qquad C_{2j+1} = b_j B, \qquad j = 1, \ldots, s-1 
\end{aligned}  
\end{equation}
Let us assume now that the symmetric scheme \eqref{2oper_sym1} is of (even) order $2p$. Taking into account \eqref{order_p} and \eqref{2oper_sym2}, then
$M(t) + M(-t) = \mathcal{O}(t^{2p})$. In consequence, application of Lemma~\ref{l:M(t)} leads to 
\[
\begin{aligned}
  \|M(t) + M(-t)\| & \le t^{2p} \, \|R_{2s+1,2p}(t) + R_{2s+1,2p}(-t) \| \\
   & \le 2 \, t^{2p} \, \sum_{k=2}^{2s+1} \int_0^1 \left\|
 \sum_{n=1}^{2p} \frac{\sigma^{n-1}}{(n-1)!} \, \ad_{C_k}^n Z_{k-1,2p-n} 
\right\|
d\sigma,
\end{aligned}
\]
and we have the following error estimate:
\begin{proposition} \label{p:2operators}
Given the time-symmetric splitting method \eqref{2oper_sym1} of order $2p$, with $A$, $B$ skew-adjoint matrices, then
\[
\| \Psi(h) - \e^{h(A+B)} \| \leq \frac{2}{2p+1} h^{2p+1} \, \sum_{k=2}^{2s+1} \int_0^1
\left\|
 \sum_{n=1}^{2p} \frac{\sigma^{n-1}}{(n-1)!} \, \ad_{C_k}^n Z_{k-1,2p-n} 
\right\|
d\sigma,
\]
where the skew-adjoint operators $Z_{k,n}$ and $C_k$ are given by \eqref{eq:Zkn} and \eqref{B_kop}, respectively.
\end{proposition}

\begin{remark}
  Whereas this result has been obtained from Lemma~\ref{l:M(t)} with $q=2p$, the order of the method, it is clear that one can get 
  more refined estimates by taking larger values of $q$, bounding appropriately the quantities $Z_{2s+1,n}$, with
  $n=2p+1, 2p+3, \ldots, q$ and finally bounding the corresponding remainder.
\end{remark}

\subsubsection{Explicit expressions for the error bounds}

We can get more insight into the error bound  by evaluating explicitly the quantity
\begin{equation} \label{integ_1}
  \mathcal{I}_{k,2p}(\sigma) \equiv  \sum_{n=1}^{2p} \frac{\sigma^{n-1}}{(n-1)!}\,
      \ad_{C_k}^{n}\, Z_{k-1,2p-n},
\end{equation}
entering the estimate provided by Proposition \ref{p:2operators}.
  To this end, we consider two sequences of integers $\{\ell(i)\}$ and $\{r(i)\}$ such that 
$\ell(i) < i$ and $r(i) < i$, and define recursively the sequence of commutators $\{E_1, E_2, E_3, \ldots\}$ formed with the operators
$A$ and $B$ as follows:
\[
E_1 = A, \qquad E_2 = B, \qquad 
E_i = [E_{\ell(i)}, E_{r(i)}], \quad i \ge 3.
\]
In addition, we consider the sequence of integers $\{\deg(i)\}$ defined recursively by
\[
   \deg(1) = 1,\qquad \deg(2) = 1,\qquad 
   \deg(i) = \deg(\ell(i)) + \deg(r(i)), \quad i\ge 3.
\]
Thus $\deg(i)$, the \emph{degree} of element $E_i$, counts the total number of operators $A$ and $B$ appearing in the
nested commutator defining $E_i$.  

Every element $F$ obtained from $A$ and $B$ through repeated commutation 
and linear combinations can then be written in the form
\[
   F = \sum_i \mu_i E_i,
\]
for a finite number of nonzero real coefficients $\mu_i$.  
Moreover, in that case 
\begin{equation}\label{eq:adAadB}
   \ad_A F = [A,F]= \sum_i \alpha_i E_i,
   \qquad
   \ad_B F = [B,F] = \sum_i \beta_i E_i,
\end{equation}
where each $\alpha_i$ and $\beta_i$ depends linearly on the coefficients 
$\mu_j$ with $j < i$ due to the triangular structure imposed by the recursion.

Table~\ref{tb:basis} lists the first values of the integer sequences $\ell(i)$ and $r(i)$ determining the elements $E_i$ of degree $\deg(i)\le 7$ for one possible choice of a basis satisfying the previously stated requirements. For this particular choice, the basis elements are ordered so that $\deg(i)<\deg(j)$ implies $i<j$. The table also includes the explicit expressions of the basis elements $E_i$ together with the corresponding coefficients $\alpha_i$ and $\beta_i$ in~\eqref{eq:adAadB}. Indices printed in bold identify those elements that remain nonzero whenever $[[[A,B],B],B]=0$; their significance will be clarified below and further discussed in the following subsection.

\begin{table}[ht]
\centering
\small
\renewcommand{\arraystretch}{1.1}
\begin{tabular}{cccclcc}
\hline
$i$ & $\deg$ & $\ell(i)$ & $r(i)$ & $E_i$ & $\alpha_i$ & $\beta_i$ \\
\hline\hline
$\mathbf{1}$ & 1 & 1 & 0 & $A$ & $0$ & $0$ \\
$\mathbf{2}$ & 1 & 2 & 0 & $B$ & $0$ & $0$ \\
\hline
$\mathbf{3}$ & 2 & 1 & 2 & $[A,B]$ & $\mu_{2}$ & $- \mu_{1}$ \\
\hline
$\mathbf{4}$ & 3 & 3 & 2 & $[[A,B],B]$ & $0$ & $- \mu_{3}$ \\
$\mathbf{5}$ & 3 & 1 & 3 & $[A,[A,B]]$ & $\mu_{3}$ & $0$ \\
\hline
$\mathbf{6}$ & 4 & 5 & 1 & $[[A,[A,B]],A]$ & $- \mu_{5}$ & $0$ \\
$7$ & 4 & 4 & 2 & $[[[A,B],B],B]$ & $0$ & $- \mu_{4}$ \\
$\mathbf{8}$ & 4 & 1 & 4 & $[A,[[A,B],B]]$ & $\mu_{4}$ & $- \mu_{5}$ \\
\hline
$\mathbf{9}$ & 5 & 6 & 1 & $[[[A,[A,B]],A],A]$ & $- \mu_{6}$ & $0$ \\
$10$ & 5 & 7 & 2 & $[[[[A,B],B],B],B]$ & $0$ & $- \mu_{7}$ \\
$\mathbf{11}$ & 5 & 5 & 3 & $[[A,[A,B]],[A,B]]$ & $0$ & $- \mu_{6}$ \\
$\mathbf{12}$ & 5 & 3 & 4 & $[[A,B],[[A,B],B]]$ & $0$ & $- \mu_{8}$ \\
$13$ & 5 & 1 & 7 & $[A,[[[A,B],B],B]]$ & $\mu_{7}$ & $- \mu_{8}$ \\
$\mathbf{14}$ & 5 & 1 & 8 & $[A,[A,[[A,B],B]]]$ & $\mu_{8}$ & $\mu_{6}$ \\
\hline
$\mathbf{15}$ & 6 & 9 & 1 & $[[[[A,[A,B]],A],A],A]$ & $- \mu_{9}$ & $0$ \\
$\mathbf{16}$ & 6 & 11 & 1 & $[[[A,[A,B]],[A,B]],A]$ & $- \mu_{11}$ & $- 2\mu_{9}$ \\
$\mathbf{17}$ & 6 & 14 & 1 & $[[A,[A,[[A,B],B]]],A]$ & $- \mu_{14}$ & $\mu_{9}$ \\
$18$ & 6 & 10 & 2 & $[[[[[A,B],B],B],B],B]$ & $0$ & $- \mu_{10}$ \\
$\mathbf{19}$ & 6 & 8 & 3 & $[[A,[[A,B],B]],[A,B]]$ & $0$ & $- 2\mu_{11} + \mu_{14}$ \\
$20$ & 6 & 3 & 7 & $[[A,B],[[[A,B],B],B]]$ & $0$ & $- \mu_{12} - \mu_{13}$ \\
$21$ & 6 & 1 & 10 & $[A,[[[[A,B],B],B],B]]$ & $\mu_{10}$ & $- \mu_{13}$ \\
$\mathbf{22}$ & 6 & 1 & 12 & $[A,[[A,B],[[A,B],B]]]$ & $\mu_{12}$ & $- \mu_{11} - \mu_{14}$ \\
$23$ & 6 & 1 & 13 & $[A,[A,[[[A,B],B],B]]]$ & $\mu_{13}$ & $- \mu_{14}$ \\
\hline
$\mathbf{24}$ & 7 & 15 & 1 & $[[[[[A,[A,B]],A],A],A],A]$ & $- \mu_{15}$ & $0$ \\
$\mathbf{25}$ & 7 & 16 & 1 & $[[[[A,[A,B]],[A,B]],A],A]$ & $- \mu_{16}$ & $- 3\mu_{15}$ \\
$\mathbf{26}$ & 7 & 17 & 1 & $[[[A,[A,[[A,B],B]]],A],A]$ & $- \mu_{17}$ & $\mu_{15}$ \\
$27$ & 7 & 23 & 1 & $[[A,[A,[[[A,B],B],B]]],A]$ & $- \mu_{23}$ & $- \mu_{17}$ \\
$28$ & 7 & 18 & 2 & $[[[[[[A,B],B],B],B],B],B]$ & $0$ & $- \mu_{18}$ \\
$\mathbf{29}$ & 7 & 11 & 3 & $[[[A,[A,B]],[A,B]],[A,B]]$ & $0$ & $- \mu_{16}$ \\
$30$ & 7 & 13 & 3 & $[[A,[[[A,B],B],B]],[A,B]]$ & $0$ & $- \mu_{19} + \mu_{23}$ \\
$31$ & 7 & 7 & 4 & $[[[[A,B],B],B],[[A,B],B]]$ & $0$ & $\mu_{20}$ \\
$\mathbf{32}$ & 7 & 8 & 4 & $[[A,[[A,B],B]],[[A,B],B]]$ & $0$ & $- \mu_{19}$ \\
$\mathbf{33}$ & 7 & 6 & 5 & $[[[A,[A,B]],A],[A,[A,B]]]$ & $0$ & $- \mu_{15}$ \\
$\mathbf{34}$ & 7 & 5 & 8 & $[[A,[A,B]],[A,[[A,B],B]]]$ & $0$ & $- \mu_{17}$ \\
$35$ & 7 & 3 & 10 & $[[A,B],[[[[A,B],B],B],B]]$ & $0$ & $- \mu_{20} - \mu_{21}$ \\
$\mathbf{36}$ & 7 & 3 & 12 & $[[A,B],[[A,B],[[A,B],B]]]$ & $0$ & $\mu_{19} - \mu_{22}$ \\
$37$ & 7 & 1 & 18 & $[A,[[[[[A,B],B],B],B],B]]$ & $\mu_{18}$ & $- \mu_{21}$ \\
$\mathbf{38}$ & 7 & 1 & 19 & $[A,[[A,[[A,B],B]],[A,B]]]$ & $\mu_{19}$ & $2(\mu_{16} - \mu_{17})$ \\
$39$ & 7 & 1 & 20 & $[A,[[A,B],[[[A,B],B],B]]]$ & $\mu_{20}$ & $- \mu_{22} - \mu_{23}$ \\
$40$ & 7 & 1 & 21 & $[A,[A,[[[[A,B],B],B],B]]]$ & $\mu_{21}$ & $- \mu_{23}$ \\
$\mathbf{41}$ & 7 & 1 & 22 & $[A,[A,[[A,B],[[A,B],B]]]]$ & $\mu_{22}$ & $\mu_{16} + \mu_{17}$ \\
\hline
\end{tabular}
\caption{Basis elements $E_i=[E_{\ell(i)}, E_{r(i)}]$ of the free Lie algebra
$\mathcal{L}(A,B)$, degree $\mathrm{deg}(i)$, and coefficients $\alpha_i$
and $\beta_i$ in~\eqref{eq:adAadB} in terms of the coefficients $\mu_j$.
Bold indices identify elements that remain nonzero whenever
$[[[A,B],B],B]= 0$.}
\label{tb:basis}
\end{table}

The operators $Z_{k,n}$ in~\eqref{eq:Zkn} and $\mathcal{I}_{2p}(\sigma)$ given by \eqref{integ_1}
can now be expressed explicitly 
in terms of the elements $\{E_i\}$ described above.  
Since each $Z_{k,n}$ is a homogeneous Lie polynomial of degree $n$ (i.e., it is a linear combination of commutators containing
exactly $n$ operators $A$ and $B$), we may write
\[
   Z_{k,n-1} \;=\; \sum_{\deg(i)=n} \theta_{k,n,i}\, E_i,
\]
for uniquely determined real coefficients $\theta_{k,n,i}$ depending on the 
splitting coefficients $a_1,b_1,\ldots,a_s,b_s$. Likewise, the function $\mathcal{I}_{k,2p}(\sigma)$, with
the operators $C_k$ given by \eqref{B_kop}, 
is a homogeneous Lie polynomial of degree $2p+1$ in $A$ and $B$, and thus admits the expansion
\[
 \mathcal{I}_{k,2p}(\sigma)   \;=\;
   \sum_{\deg(i)=2p+1} \omega_{k,2p+1,i}(\sigma)\, E_i,
\]
for suitable coefficient functions $\omega_{k,n,i}(\sigma)$, $n=2p+1$.

With these representations, the error bound of 
Proposition~\ref{p:2operators} becomes
\begin{equation} \label{bound_explicit}
\begin{aligned}
   \|\Psi(h) - \e^{h(A+B)}\|
   &\;\le\; \frac{2}{2p+1}\, h^{2p+1}
      \sum_{k=2}^{2s+1} \int_0^1
      \left\|
         \sum_{\deg(i)=2p+1} \omega_{k,2p+1,i}(\sigma)\, E_i
      \right\| d\sigma \\
   &\;\le\;  h^{2p+1}
      \sum_{\deg(i)=2p+1}
      \left(\frac{2}{2p+1}\,
       \sum_{k=2}^{2s+1} 
        \int_0^1 |\omega_{k,2p+1i}(\sigma)|\, d\sigma
      \right)
      \|E_i\|,
\end{aligned}
\end{equation}
whereas the order conditions for a time-symmetric splitting method \eqref{2oper_sym1} of order $2p$
are simply
\begin{equation} \label{or_cond}
  \theta_{2s+1,n,i} = 0, \qquad n = 1, 3, \ldots, 2p-1, \qquad \deg(i) = n.
\end{equation}  
For any given sequence of real coefficients $a_1,b_1,\ldots,a_s,b_s$, 
each $\theta_{k,n,i} \in \mathbb{R}$ can be computed
recursively from~\eqref{eq:Zkn}. Such computations are straightforward for  
$n\le 7$ using~\eqref{eq:adAadB}  with the information provided in 
Table~\ref{tb:basis}).  
Furthermore, each $\omega_{k,n,i}(\sigma)$ is a real polynomial in $\sigma$ 
of degree at most $2p-1$ and can likewise be obtained recursively from \eqref{eq:Zkn} for $2p+1\leq 7$.  
For $n>7$, the coefficients $\theta_{k,n,i}$ and the real polynomials  $\omega_{k,n,i}(\sigma)$ can be computed recursively with an extended table of rules for the coefficients $\alpha_i$ and $\beta_i$ featuring in~\eqref{eq:adAadB}.  
Thus, in the \emph{Mathematica} notebook accompanying this article, and available at
\begin{center}
  \url{http://www.gicas.uji.es/Research/EBSplitting.html}
\end{center}  
 we include,
in addition to data analogous to Table~\ref{tb:basis} up to degree $11$,
a set of lightweight routines implementing the required recursive procedure.

The error bound \eqref{bound_explicit} has an additional important advantage over the more general 
bounds obtained in Section~\ref{sec.2}: it yields significantly sharper estimates in 
situations where certain commutators vanish for structural reasons.  Thus,
in many relevant applications one encounters the identity
\begin{equation} \label{rkn_1}
   [[[A,B],B],B] = 0.
\end{equation}
This is the case, for instance, of classical Hamiltonian systems whose kinetic energy is quadratic in momenta and the potential
only depends on coordinates, and also the time-dependent Schr\"odinger equation, to be discussed below.
In such cases, all basis elements $E_i$ involving this 
commutator or any of its descendants vanish identically, and therefore the 
corresponding contributions $\theta_{k,n,i}$ and $\omega_{k,n,i}(\sigma)$ do not 
appear in the local error expansion.  As a result, the bounds simplify 
substantially and become considerably sharper.

From an algebraic point of view,
the sequences of indices $\ell(i)$ and $r(i)$ used in Table~\ref{tb:basis} 
encode a basis of the free Lie algebra $\mathfrak{L}(A,B)$ generated by $A$ and $B$.  
More precisely, the elements listed in Table~\ref{tb:basis} constitute the 
initial segment, up to degree $\le 7$, of the basis for the free Lie algebra 
proposed in~\cite{mclachlan19tla}.  
This basis is a particular instance of a Hall--Viennot--Sirsov basis 
\cite{viennot78adl,reutenauer93fla}, constructed so as to enjoy a triangular recursive 
structure compatible with the adjoint actions~\eqref{eq:adAadB}.  
An additional useful feature of the basis introduced in~\cite{mclachlan19tla} is that it yields a natural direct 
sum decomposition of the free Lie algebra
\[
   \mathfrak{L}(A,B) \;=\; I \;\oplus\; \mathfrak{L}(A,B) / I,
\]
where $I$ is the Lie ideal generated by the element $[[[A,B],B],B]$.  
In other words, the basis elements split into
\begin{itemize}
   \item those that lie in $I$ (and hence vanish for operators satisfying 
         $[[[A,B],B],B]=0$), and
   \item those that project to nonzero elements in the quotient 
         $\mathfrak{L}(A,B)/I$.
\end{itemize}

The indices typeset in \textbf{boldface} in Table~\ref{tb:basis} identify 
precisely the elements of the basis of the quotient $\mathfrak{L}(A,B)/I$, 
i.e., the Lie monomials that remain nonzero whenever $[[[A,B],B],B]= 0$.  
The indices in normal (nonbold) type correspond to elements of the ideal $I$, 
and hence to commutators that vanish automatically in any setting where 
$[[[A,B],B],B]=0$.  

\subsubsection{Numerical illustration}

The estimate \eqref{bound_explicit} for the local error bound is illustrated next on some standard schemes applied to a system
split into two parts such that condition \eqref{rkn_1} holds.

As a first example
we take again the 6th-order time-symmetric composition  \eqref{sofro_spa} of order 6 and 13 stages
already considered  in section \ref{sec.3.4} with 
\[
    S_2(\gamma_j h) = \e^{\frac{h}{2} \gamma_j A} \, \e^{h \gamma_j B} \, \e^{\frac{h}{2} \gamma_j A}.
\]
By carrying out the necessary calculations we arrive at    
\[
\begin{aligned}
  & \|\Psi(h) - \e^{h(A+B)}\|  \le  10^{-5} \cdot h^7 \Big(   0.0865  \, \|E_{24}\| +  0.6488  \, \|E_{25}\| + 0.5591 \, \|E_{26}\| \\
      & \qquad\qquad  + \underline{1.4199 \, \|E_{27}\|} + \underline{0.2866 \, \|E_{28}\|}  + 0.4072  \, \|E_{29}\| + \underline{1.9947  \, \|E_{30}\|} \\
     & \qquad\qquad + \underline{1.1001 \,  \|E_{31}\|} + 0.8357 \, \|E_{32}\| + 0.1288 \, \|E_{33}\| + 0.6438  \, \|E_{34}\| \\
     & \qquad\qquad + \underline{2.002 \, \|E_{35}\|} + 1.6287 \, \|E_{36}\| + \underline{1.1208 \, \|E_{37}\|} +   2.6227 \, \|E_{38}\| \\
     & \qquad\qquad  + \underline{2.6531 \, \|E_{39}\|} + \underline{1.7839 \, \|E_{40}\|} +  0.7644  \, \|E_{41}\|  \Big).
\end{aligned}
\]     
\underline{Underlined} terms disappear from the error bound when the identity $  [[[A,B],B],B] = 0$ is satisfied.

Next we take the two well known Runge--Kutta--Nystr\"om splitting methods of order 4 and 6 presented in \cite[Table 3]{blanes02psp}
and recommended in \cite{blanes24smf} for their high efficiency. The error bounds are given, respectively, by
\[
   \|\Psi(h) - \e^{h(A+B)}\|  \le  10^{-3} \cdot h^5 \Big( 10.15 \, \|E_{9}\| + 1.04 \, \|E_{11}\| + 0.022 \, \|E_{12}\| 
 + 3.56 \, \|E_{14}\| \Big)
\]
and
\[
\begin{aligned}
  & \|\Psi(h) - \e^{h(A+B)}\|  \le  10^{-7} \cdot h^7 \Big(   0.13146  \|E_{24}\| +  0.46785  \|E_{25}\| + 0.63853 \|E_{26}\| \\
     & \qquad\qquad + 0.69042  \|E_{29}\| + 1.59218 \|E_{32}\| + 0.05566 \|E_{33}\| + 0.924511  \|E_{34}\| \\
     & \qquad\qquad + 3.85208  \|E_{36}\| + 3.81  \|E_{38}\| + 0.1649  \|E_{41}\|  \Big).
\end{aligned}
\]     
Observe that the individual terms in the error bound for this specially designed scheme attain significantly smaller numerical values than those corresponding to the composition \eqref{eq:palindromic_gamma}, in agreement with the behavior observed in practical integrations.

\subsection{Estimates for unbounded skew-adjoint operators}

The error bound provided by Proposition \ref{p:2operators} (or alternatively by eq. \eqref{bound_explicit})
has been obtained with the requirement that $A$ and $B$ are bounded. It turns out, however, that it can be extended (under suitable
general assumptions) to the case
when $A$ and possibly $B$ are unbounded 
skew-adjoint operators acting on
a Hilbert space $\mathcal{H}$. Thus, suppose we have a splitting method \eqref{2oper_sym1} which is of order $2p$ in the classical, bounded setting (i.e., it satisfies eq. \eqref{or_cond}) and that we want to bound the local error
 $\|\Psi(h) u_0 - \e^{h(A+B)} u_0 \| $ for some $u_0 \in \mathcal{H}$.
 We make
the following assumptions on $u_0$ and the operators $A$ and $B$:

\begin{itemize}
  \item \textbf{Assumption 1}. $A$, $B$, and $A+B$ are skew-adjoint operators on $\mathcal{H}$, and $u_0$ belongs
  to the domain of $A$, $B$ and $A+B$.
  \item \textbf{Assumption 2}. $G$ is a group of unitary operators on $\mathcal{H}$ containing $\e^{s A}$, $\e^{s B}$,
and $\e^{s (A+B)}$ for all $s \in \mathbb{R}$.
  \item \textbf{Assumption 3}. There exists a $G$-invariant dense subspace $\mathcal{D}(G)$  in $\mathcal{H}$ such that for $u_0 \in \mathcal{H}$ and for
all $g \in G$, then $g \, u_0 \in \mathcal{D}(G)$ and  the Lie elements $E_i$ with $\deg(i) = 2p+1$ 
are well-defined on $\mathcal{D}(G)$.
  \item \textbf{Assumption 4}. There exist constants $d_{i} \geq 0$ such that  $\sup_{g \in G} \| E_i \, g \, u_0\| \leq d_{i}$, with
  $\deg(i) = 2p+1$. 
\end{itemize}
From the proof of Lemma \eqref{l:Theta_M(t)} we have 
\[
  \|\Psi(h) u_0 - \e^{h(A+B)} u_0 \| \le \int_0^h \| (M(t) + M(-t)) \, \gamma u_0\| dt,
\]
with
\[
 \gamma = \Theta(-t) \, \e^{\frac{h}{2}(A+B)} \in G \qquad \mbox{ and } \qquad  
 M(t) + M(-t) =  t^{2p}  \big(R_{2s+1,2p}(t) + R_{2s+1,2p}(-t) \big).
\]
On the other hand,
\[
  \|R_{2s+1,2p}(t) \gamma u_0 \|   \le \sum_{k=2}^{2s+1} \left\| \sum_{n=1}^{2p} G_n(t, C_k, Z_{k-1,2p-n}) \gamma u_0 \right\| 
   \le  \sum_{k=2}^{2s+1} \int_0^1 \left\| \mathcal{I}_{k,2p}(\sigma) \, \gamma \, u_0 \right\| d\sigma,  
\]     
so that 
\begin{equation} \label{local_error_unb}
  \left\|\Psi(h) u_0 - \e^{h(A+B)} u_0 \right\| 
     \;\le\; 
     h^{2p+1}\,  \sum_{\deg(i)=2p+1}
      \left(
      \frac{2}{2p+1}\, 
       \sum_{k=2}^{2s+1} 
        \int_0^1 |\omega_{k,i}(\sigma)|\, d\sigma
      \right) d_i.
\end{equation}
%

\paragraph{Time-dependent Schr\"odinger equation.}

The time-dependent Schr\"odinger equation
\[
  i \psi(x,t) = -\frac{1}{2} \Delta \psi(x,t) + V(x) \psi(x,t), \qquad \psi(x,0) = \psi_0
\]
can be interpreted as an initial value problem of the form \eqref{equ_2op} by setting $u(t) = \psi(\cdot,t)$ and
\[
  (A w)(x) = \frac{i}{2} \Delta w(x), \qquad (Bw)(x) = -i V(x) w(x)
\]
for a a sufficiently regular function $w: \Omega \subset \mathbb{R}^n \longrightarrow \mathbb{C}$. In that case,
$\mathcal{H} = L^2(\Omega, \mathbb{C})$, $A$ is clearly an unbounded linear operator, the previous Assumptions 1 and 2
are satisfied and the Lie bracket
$[[[A,B],B],B] w(x) =0$ for all $w(x)$. If in addition the initial condition is such that Assumptions 3 and 4 hold, then the previous
bound \eqref{local_error_unb} applies, with elements $E_i$ whose subindex $i$ is typeset in boldface in Table \ref{tb:basis}, thus
leading to improved constants in the local error estimate.

\section{Concluding remarks}

In this work we have carried out a detailed error analysis of splitting and
composition methods applied to finite-dimensional unitary problems.
Explicit bounds have been derived both in terms of the norms of the
operators involved and in terms of the norms of their nested commutators.
We have considered splittings involving an arbitrary number of operators, and then 
we specialized the discussion to the case
of two operators and, in particular, to time-symmetric schemes. In this
setting, sharper bounds were obtained by exploiting a basis of the free
Lie algebra generated by the two operators, leading to tighter constants
and a more refined description of the structure of the local error.

Although we have assumed that
the operators are bounded, the formulation of the error bounds in terms
of commutators allows for extensions to situations where one of the
operators is unbounded, provided suitable domain and regularity
assumptions are satisfied. This framework is especially relevant for the
time-dependent Schr\"odinger equation, where the intrinsic algebraic and
unitary structure leads to improved constants in the corresponding error
estimates.

The results obtained here could be useful in several contexts. From the
perspective of numerical analysis of differential equations, they provide
a systematic tool for assessing the relative efficiency of different
splitting and composition methods. This feature has been illustrated here for a specific class of composition methods.
 The explicit characterization of the
error constants also offers guidance in the construction of new schemes,
allowing for the selection of coefficients that minimize dominant error
terms. In the context of quantum simulation, the bounds derived in this
work supply quantitative information and explicit values for the error
constants, which is essential for reliability estimates and resource
assessment in practical implementations.

Several directions for future research naturally emerge from this study.
A detailed treatment of problems involving genuinely unbounded operators
would require a careful analysis of domain invariance and regularity
conditions, and would extend the applicability of the present results to
a broader class of evolution equations. Another promising direction is
the investigation of splitting methods with complex coefficients in the
unitary and parabolic settings, where stability and structural properties
play a crucial role. Extensions to schemes involving commutators
explicitly, as well as to nonlinear problems, also constitute natural
continuations of this work. In the latter case, the Lie formalism for
ordinary differential equations provides a formal linear framework that
may serve as a starting point for a rigorous analysis. Finally, in the
area of quantum simulation, it would be of interest to apply the present
theoretical bounds to concrete models arising in practice, thereby
obtaining problem-dependent performance estimates and further insight
into the efficiency of splitting-based algorithms. All these issues require further study and will be addressed in a subsequent work.

\subsection*{Acknowledgements}
This work has been supported by 
Ministerio de Ciencia e Innovaci\'on (Spain) through projects  PID2022-136585NB-C21 and PID2022-136585NB-C22, \newline
MCIN/AEI/10.13039/501100011033/FEDER, UE. 

\bibliographystyle{agsm}

\end{document}